\documentclass[12pt]{article}
\usepackage{graphicx, amsmath, amssymb, cite, setspace, color, relsize, bm, bbold, array, hyperref, dsfont, bbold, cancel,soul,placeins,hyperref,amsfonts,pifont,comment}
\usepackage{wasysym}
\usepackage[dvipsnames]{xcolor}

\usepackage[top=1 in, bottom=1 in, left=1 in, right=1 in]{geometry}

\usepackage{mathtools}
\usepackage{amsthm}
\usepackage{amssymb}
\usepackage[alphabetic]{amsrefs}
\usepackage{graphicx}
\usepackage{tikz}
	\usetikzlibrary{decorations.pathreplacing}
	\usetikzlibrary{patterns}
\usepackage{caption}
\usepackage{enumitem}

\newcommand{\ra}{\rightarrow}
\newcommand{\pr}{\prime}
\newcommand{\de}{\partial}

\newcommand{\R}{\mathbb{R}}

\newcommand{\lbar}[1]{\overline{#1}}
\DeclareMathOperator{\id}{id}

\DeclareMathOperator{\U}{U}

\newtheorem{thm}{Theorem}
\newtheorem{cor}{Corollary}
\theoremstyle{definition}

\newtheorem*{defin*}{Definition}
\theoremstyle{plain}
\newtheorem{lemma}{Lemma}
\newtheorem{innercustomthm}{Theorem}
\newenvironment{customthm}[1]
  {\renewcommand\theinnercustomthm{#1}\innercustomthm}
{\endinnercustomthm}
\newtheorem{innercustomlm}{Lemma}
\newenvironment{customlm}[1]
  {\renewcommand\theinnercustomlm{#1}\innercustomlm}
{\endinnercustomlm}

\newcommand{\captionfonts}{\footnotesize} 
\makeatletter
\long\def\@makecaption#1#2{%
  \vskip\abovecaptionskip
  \sbox\@tempboxa{{\captionfonts #1: #2}}%
  \ifdim \wd\@tempboxa >\hsize
    {\captionfonts #1: #2\par}
  \else
    \hbox to\hsize{\hfil\box\@tempboxa\hfil}%
  \fi
  \vskip\belowcaptionskip}
\makeatother

\def\lsim{ \lower .75ex \hbox{$\sim$} \llap{\raise .27ex
\hbox{$<$}} }
\def\gsim{ \lower .75ex \hbox{$\sim$} \llap{\raise .27ex
\hbox{$>$}} }

\usepackage{eso-pic}
\usepackage{lipsum}

\let\oldsqrt\sqrt
\def\sqrt{\mathpalette\DHLhksqrt}
\def\DHLhksqrt#1#2{%
\setbox0=\hbox{$#1\oldsqrt{#2\,}$}\dimen0=\ht0
\advance\dimen0-0.2\ht0
\setbox2=\hbox{\vrule height\ht0 depth -\dimen0}%
{\box0\lower0.4pt\box2}}

\begin{document}

\title{Enhanced Bishop-Gromov Theorem}
\date{} 

\author{Adam R.~Brown$^{a,b}$ and Michael H.~Freedman$^{c,d}$}
\maketitle
\vspace{-5mm}
\begin{centering}
$^{a}$ {Google Research (Blueshift), Mountain View, California}

$^{b}$ {Physics Department, Stanford University, Stanford, California }

$^{c}$ {Microsoft Research, Santa Barbara, California}

$^{d}$ {Mathematics Department, UC Santa Barbara, California }

\end{centering}

\begin{abstract}
\noindent The Bishop-Gromov theorem upperbounds the rate of growth of volume of geodesic balls in a space, in terms of the most negative component of the Ricci curvature. In this paper we prove a strengthening of the Bishop-Gromov bound for homogeneous spaces. Unlike the original Bishop-Gromov bound, our enhanced bound depends not only on the most negative component of the Ricci curvature, but on the full spectrum. As a further result, for finite-volume inhomogeneous spaces, we prove an upperbound on the average rate of growth of geodesics, averaged over all starting points; this bound is stronger than the one that follows from the Bishop-Gromov theorem. Our proof makes  use of the Raychaudhuri equation, of  the fact that geodesic flow conserves phase-space volume, and also of a tool we introduce for studying families of correlated Jacobi equations that we call ``coefficient shuffling''. 
\end{abstract}

\tableofcontents 
\newpage 
\section{Introduction}

\subsection{Introduction for physicists}

Pick a point in a Riemannian manifold, and shoot out geodesics in every direction. As time passes, these geodesics will have explored a larger and larger region. This paper will be concerned with the following question
\begin{equation}
\textrm{\bf {question}}: \ \ \ \textrm{as a function of time, what is the volume explored by the geodesics?} 
\end{equation}
In particular, this paper will construct a novel upperbound on this volume as a function of the Ricci curvature. \\

\noindent Let's begin with the simplest nontrivial examples, the maximally symmetric two-dimensional spaces. Since these spaces are homogeneous, it doesn't matter where we start. Up to rescalings, there are  three possibilities, determined by the sign of the curvature:
\begin{eqnarray}
\mathbb{H}^2: &  ds^2 = dt^2 + \sinh^2 \hspace{-1pt} t \, d \theta^2 & \rightarrow \ \ \textrm{vol(t)} = 2 \pi \left(  \cosh  t - 1 \right) \label{eq:H2volumeoftime} \\
\mathbb{R}^2: & \hspace{-.78cm} ds^2 = dt^2 + t^2 \, d \theta^2 & \rightarrow \ \ \textrm{vol(t)} = \pi t^2 \\
\mathbb{S}^2: &  \hspace{-.21cm}  ds^2 = dt^2 + \sin^2 \hspace{-1pt} t \, d \theta^2 & \rightarrow \ \ \textrm{vol(t)} =  2 \pi \left( 1 - \cos t \right)  \ \ \ \ \ \textrm{for } t \leq \pi \  . \label{eq:S2volumeoftime}  
\end{eqnarray}
We see that geodesic balls grow fastest in the negatively curved space, $\mathbb{H}^2$, and slowest in the positively curved space, $\mathbb{S}^2$. The reason is clear: the more negative the curvature, the more the geodesics diverge; and the more geodesics diverge, the more space they consume. This will be an important principle, 
\begin{equation}
\textrm{\bf {heuristic}}: \ \ \ \textrm{negative curvature makes volumes grow fast} \label{eq:heuristicprinciple} \ . 
\end{equation}
     \begin{figure}[htbp] 
    \centering
    \includegraphics[width=2.5in]{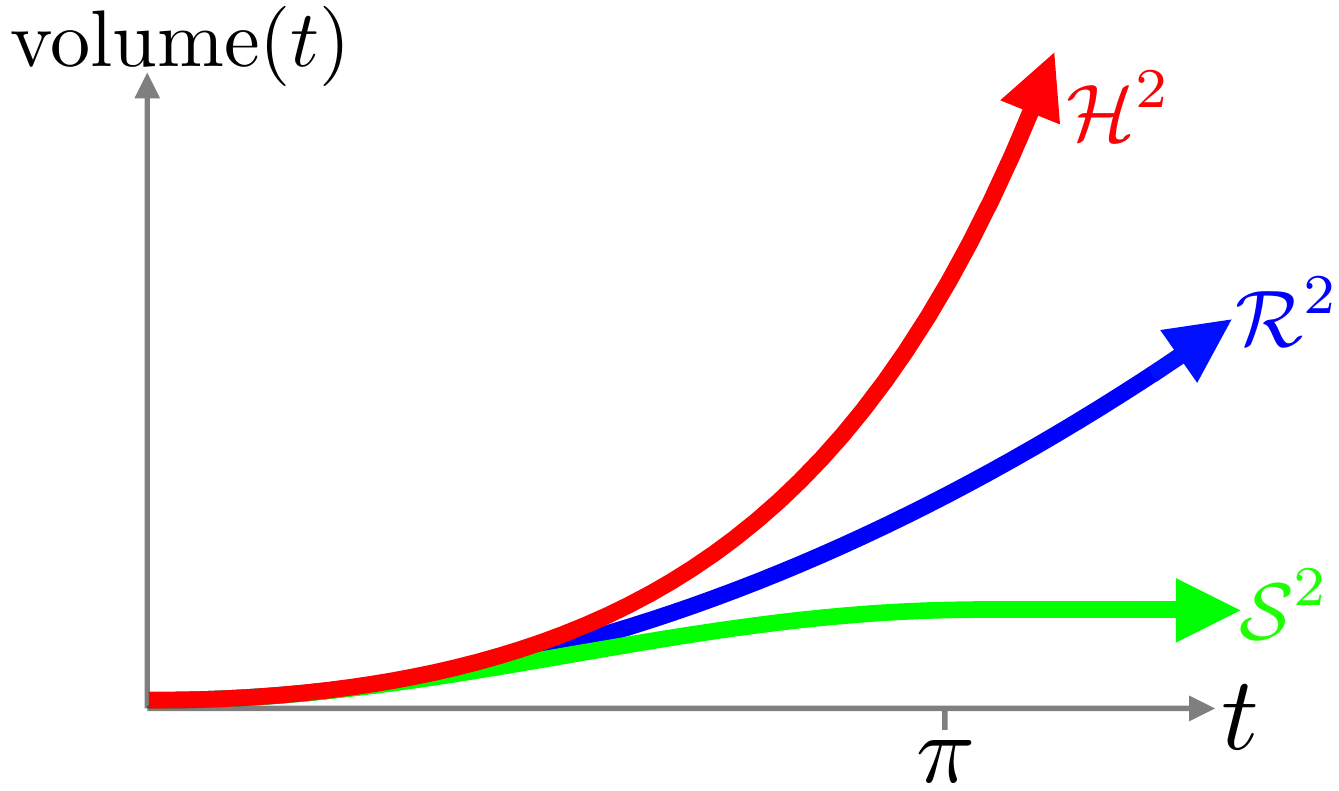} 
    \caption{The volume of a geodesics ball as a function of time for the unit maximally symmetric two-dimensional spaces, given by Eqs.~\ref{eq:H2volumeoftime}-\ref{eq:S2volumeoftime}. The more negative the curvature, the faster the volume grows, exemplifying the heuristic principle~\ref{eq:heuristicprinciple}. At $t=\pi$, the geodesics on the two-sphere reach a cut locus, there is no space left to explore, and the volume saturates at $4\pi$.}
 \end{figure}
 \newpage
The foundational result in trying to make this heuristic precise was proved in 1963, and is known as the Bishop-Gromov theorem \cite{BishopGromov}. (Gromov's name is usually appended for unpublished refinements, see \cite{ballman}.)   For a general $d$-dimensional Riemannian metric, calculating the volume of a geodesic ball as a function of time will in general be quite laborious. However, Bishop and Gromov showed that there is a simple upperbound. Let's describe this bound,  and our enhancement of it. 
\subsubsection{Bishop-Gromov bound for homogeneous spaces}
A homogeneous space is one in which all points (but not necessarily all directions) are the same. In a homogeneous space, therefore, the volume of a geodesic ball is independent of the starting point. The BG theorem upperbounds this volume. First find the most negative component of the Ricci curvature tensor $\mathcal{R}_{\mu \nu} X_\textrm{min}^{\mu} X_\textrm{min}^{\nu}$, which means minimizing the quantity $\mathcal{R}_{\mu \nu} X^{\mu} X^{\nu}$ over all unit tangent vectors $X^{\mu}$.  Next, consider a new space---the maximally symmetric space of the same dimension for which every principal component of the curvature tensor is equal to the $\mathcal{R}_{\mu \nu} X_\textrm{min}^{\mu} X_\textrm{min}^{\nu}$ of the original space. Then calculate the volume of a geodesic ball in this new space, 
\begin{equation}
\textrm{BG}(t)  \equiv   \Omega_{d-1} \, \int_0^t d\tau \, \operatorname{sn}\Bigl(  \frac{{ \mathcal{R}_{\mu \nu} X_\textrm{min}^{\mu} X_\textrm{min}^{\nu}}}{{d-1}}, \tau \Bigl)  ^{d-1} , \label{eq:definitionofBGoft}
\end{equation}
where $\Omega_{d-1}$ is the area of a unit $d$-1-sphere, so that $\Omega_{1} = 2 \pi, \Omega_2 = 4 \pi, \Omega_3 = 2 \pi^2$, etc., and 
\begin{equation}
\operatorname{sn}(k,t)  \equiv 
\left\{ \begin{array}{ccccl} 	
\frac{\sin(\sqrt{k}t)}{\sqrt{k}} & \textrm{for} & k>0& \& & 0 \leq t \leq \frac{\pi}{\sqrt{k}}\\
0 & \textrm{for} & k>0& \& & \ \, \ \ \  \  t \geq  \frac{\pi}{\sqrt{k}}\\
t & \textrm{for} & k=0& \& & \ \ \ \ \ \, t \geq 0\\
 \frac{\sinh(\sqrt{-k}t)}{\sqrt{-k}} & \textrm{for} & k<0& \& & \ \ \ \ \ \, t \geq 0  \ . 
 \end{array} \right.   \label{eq:definitionofsnfunction}
\end{equation}
The Bishop-Gromov theorem says that the volume explored by a geodesic ball in the original space is smaller than the volume of a geodesic ball in the maximally symmetric space,
\begin{equation}
\textrm{BG bound}: \ \ \textrm{volume}(t) \ \leq \  \textrm{BG}(t) \  .  \label{eq:astatementofBGbound}
\end{equation}
(This version is due to Bishop \cite{BishopGromov}, later monotonicity inequalities were proved by Gromov.) A \emph{lower}bound on the Ricci curvature gives an \emph{upper}bound on the volume. 

\subsubsection{Enhanced Bishop-Gromov bound for homogeneous spaces}
One inefficiency in the Bishop-Gromov bound is that it treats all directions the same: if one direction has large negative curvature, it treats all directions as though they have large negative curvature.  This form of collective punishment leads to the upperbound being loose for highly anisotropic spaces. Our new bound will manage to be tighter by treating the different directions more individualistically. Define the quantity 
\begin{equation}
\textrm{enhanced-BG}(t) 
  \equiv    \int_0^t d\tau  \int d \Omega_{d-1} \, \operatorname{sn}\Bigl(  \frac{{ \mathcal{R}_{\mu \nu} X_{\Omega}^{\mu} X_{\Omega}^{\nu}}}{{d-1}}, \tau \Bigl)^{d-1} , \label{eq:explicitimproveBGversion2}
\end{equation}
where $X^\mu_\Omega$ is the unit tangent vector that leaves the origin on a bearing $\Omega$. All the quantities in this expression are to be evaluated at the point of departure, so that $\mathcal{R}_{ \mu \nu } X_\Omega^\mu X_\Omega^\nu$ is a function only of the original angle $\Omega$ and not of time $\tau$; this means that to evaluate enhanced-BG($t$) we do not need to able to solve the geodesic equation, we only need to be able to do an angular integral over the eigenvalues of the Ricci tensor. (For an example of this quantity evaluated for an explicit metric, see Sec.~\ref{sec:improvedBGexamples}.) 
For homogeneous spaces, we will prove that this upperbounds the volume growth of geodesic balls, 
\begin{equation}
\textrm{new theorem}: \ \ \textrm{volume}(t) \leq \textrm{enhanced-BG}(t) \ . \hspace{1cm} \label{eq:newtheorem}
\end{equation}
We will formalize this below as theorem 5. This new bound is stronger than the BG bound 
\begin{equation}
\textrm{enhanced-BG}(t) \ \leq \ \textrm{BG}(t) \ ,
\end{equation} 
with equality only for Einstein spaces (i.e.~only when every eigenvalue of the Ricci tensor is the same, so that $\mathcal{R}_{\mu \nu} = \textrm{constant} \times g_{\mu \nu}$). Just as with the original Bishop-Gromov theorem, this gives a simple upperbound on the volume of geodesic balls; the only information we need to evaluate the bound for homogeneous spaces is the eigenvalues of the Ricci curvature tensor. The original BG bound took the worst-case most-expansive direction $X^\mu$, whereas the enhanced-BG bound considers a time-dependent weighted-average over all angles.

\subsubsection{Bishop-Gromov bound for inhomogeneous spaces}
For inhomogeneous spaces, the volume($t,q$) of a geodesic ball may depend not only on its radius $t$, but also on the point $q$ from which it emanates. The Bishop-Gromov theorem upperbounds the volume of any geodesic ball, no matter what its starting point. The upperbound is defined by taking the maximum value of Eq.~\ref{eq:definitionofBGoft} over all possible starting points, 
\begin{equation}
\textrm{volume}(t,q) \ \leq \ \textrm{max}_q 
  \Omega_{d-1} \, \int_0^t d\tau  \, \operatorname{sn}\Bigl(  \frac{{ \mathcal{R}_{\mu \nu} (q) X_\textrm{min}^{\mu}(q)X_\textrm{min}^{\nu}(q)}}{{d-1}}, \tau \Bigl)^{d-1} . \label{eq:BGforinhomogeneous111}
\end{equation}
 The BG bound for inhomogeneous spaces thus makes two different worst-case assumptions: it considers the worst-case most-expansive direction $X^\mu_\textrm{min}$, and it considers the worst-case most-expansive starting point $q$. 

\subsubsection{New bound for inhomogeneous spaces}

Unfortunately, for inhomogeneous spaces we will not be able to improve the bound on the quantity $\textrm{volume}(t,q)$. For the worst-case starting point, the regular BG-bound Eq.~\ref{eq:BGforinhomogeneous111} is still as good as we can do. Instead, we will develop a new bound on a different quantity---the average volume growth, averaged over all starting points. (In order for this average to be well-defined, we will initially need to restrict ourselves to considering spaces of finite volume.) Our bound will be 
\begin{equation}
\int dq \ \textrm{volume}(t,q) \ \leq \ \int dq \ 
  \Omega_{d-1} \, \int_0^t d\tau  \, \operatorname{sn}\Bigl(  \frac{{ \mathcal{R}_{\mu \nu} (q) X_\textrm{min}^{\mu}(q)X_\textrm{min}^{\nu}(q)}}{{d-1}}, \tau \Bigl)^{d-1} . 
\end{equation}
(Here `$dq$' is to be understood as telling us to integrate with respect to the volume-form on the space.) It is clear that this is a tighter bound than could be derived simply by integrating Eq.~\ref{eq:BGforinhomogeneous111} over $q$, and is therefore tighter than simply follows from the regular BG bound. We will formalize this statement in Sec.~\ref{sec:BGfinite}. For infinite-volume inhomogeneous spaces, the averaging procedure is more delicate, and we have to settle for making the weaker statements formalized in Sec.~\ref{sec:BGinfinite}.

\subsection{Introduction for mathematicians} 

A theorem, originally due to Bishop \cite{bishop}, later extended to the Bishop-Gromov inequalities, provides an upper bound on volume growth in terms of Ricci curvature. It says, in simplest form, that in any Riemannian manifold $M^d$ the volume of the ball $B_q^M(r)$ of radius $r$ about a point $q$ grows no faster than the corresponding ball in a maximally symmetric space $H^d$ of the same dimension $d$ and scaled to have the (unique) eigenvalue of its Ricci tensor $\operatorname{Ric}(H)$ equal to the \emph{smallest} eigenvalue of the Ricci tensor $\operatorname{Ric}_q(M)$, over all $q \in M$,
\begin{equation}\label{eq:vol}
	\operatorname{Vol}(B_q^M(r)) \leq \operatorname{Vol}(B^{{H}}(r)) \ . 
\end{equation}
We reconsider this, and related inequalities, for quantities averaged over $M$. When $M$ is homogenous the new inequalities are stronger than their classical counterparts.

If $\operatorname{Ric}_q(M)$ has a broad spectrum one may be dissatisfied with the classical inequality (\ref{eq:vol}) since the r.h.s.~notices only the smallest eigenvalue of $\operatorname{Ric}_q(M)$ (over all $q \in M$) and thus may give a needlessly high upper bound. It is often observed that one cannot simply average the spectrum of $\operatorname{Ric}_q(M)$ to produce Ricci scalar curvature, and expect a similar inequality. The standard example is a hyperbolic space cross an $n$-sphere ($n \geq 2$) of sufficiently small radius. This space can have positive scalar curvature \emph{but} exponential volume growth. (And we cannot rescue such a bound by restricting to non-positive curvature, as is exhibited in the appendix by considering the example of $\mathbb{H}^3 \times \mathbb{R}^2$.) The purpose of this paper is to prove an inequality similar to (\ref{eq:vol}) which takes the entire spectrum of the family of the Ricci tensors $\operatorname{Ric}_q(M)$ into account. First we will consider the infinite volume case and then turn our attention to finite volume where the Hamiltonian nature of geodesic flow gives us access to certain averaged quantities. An important feature of our inequality is that (for finite volume $M$) it addresses the average growth around all points rather than from a particular point. In the case of homogeneous manifolds, the volume of a geodesic ball is independent of the point of departure $q$, so the growth from any point is given by the average growth. 

An exciting trend in geometry in recent decades has been an openness to studying manifolds of very high dimension. Certainly computer science has been a consumer of high-dimensional geometry. One may expect that the most interesting applications of our result will be in high dimensions where saying $\operatorname{Ric}_q(M)$ has broad spectrum takes on meaning.

Quantum computing provides a motivating example, explored in \cite{bfls}, for how geometries in very high dimension unavoidably arise. This point of view builds on a picture introduced by Nielsen and collaborators \cite{Nielsen1,Nielsen2} 
 and \cite{bs} that the symmetry group $\operatorname{SU}(2^N)$ of an $N$-qubit system should be given a left invariant ``penalty'' metric whose geodesic geometry encodes efficient computation. Penalty metrics, by definition, reflect the cost, exponential in $k$, of controlled motion in $\operatorname{SU}(2^N)$ in a $k$-body direction, $1 \leq k \leq N$. In this way the ball of radius $r$ about $\id \in \operatorname{SU}(2^N)$, $B_{\id}^{\text{penalty}}(r)$, is a surrogate\footnote{This relies on the assumption that quantum computing is the ultimate (realistic) physical model. And that differences between continuous and discrete dynamics is not fundamental.} for what might ever be computed, $r$ being some function of time, space, and energy resources. The volume growth of $B_{\id}^{\text{penalty}}(r)$ is certainly of epistemological\footnote{Mathematicians  rightly shun big (sesquipedalian) words but in this case, the fit seems exact.} interest. For a quantum computer of a million qubits, $\dim(\operatorname{SU}(2^N)) = 4^{10^6} - 1$. As black hole dynamics is now understood as quantum computation \cite{Susskind:2014rva}, and a galactic core black hole may have $10^{90}$ degrees of freedom, geodesic geometry on a $4^{10^{90}}$-dimensional Lie group is relevant to black hole evolution. It is with an eye to such extravagant dimensions that we will tease out the influence of the full spectral content of $\operatorname{Ric}_q(M)$.

After developing some machinery and proving our theorems, we will explain in which regime they are strongest. As it will be seen, their greatest force is intermediate time scales; at the longest time scales they can only improve the usual BG inequalities by a subexponential multiplicative factor. 

This paper introduces a new technique to the study of families of Jacobi equations, second-order linear ODEs. The families we study are indexed by a probability measure space $T$. The various ODEs indexed by $\tau \in T$ are \emph{not coupled}, but their coefficient functions $\kappa_\tau(t)$ are \emph{correlated} in a manner to be explained. Our goal is to upper bound the $T$-\emph{average} growth rate of the solution $j_\tau(t)$. The method we develop, coefficient shuffling, allows a comparison to an exactly solvable, constant coefficient, ``version'' of the original family of equations. This method is also useful in estimating averages under various geometric flows, for example the Raychaudhuri equation from general relativity. Note that coefficient shuffling is an analytic trick that produces a new collection of Jacobi equations, but these are merely convenient fictions that do not generally correspond to geodesic flow on any Riemannian manifold. Thus if our theorem is thought of a ``comparison'' theorem, we are comparing to a fictitious manifold.

\subsection{Heuristic physics proof overview}  \label{sec:heuristicphysicsproof}
In Sec.~\ref{sec3} we will prove our theorem in the language of mathematics, but first let's give a heuristic overview of the proof using the vocabulary of physics. We will consider the homogeneous case, covered by Eq.~\ref{eq:newtheorem}. 

 An equation famous to physicists is the Raychaudhuri equation \cite{Raychaudhuri:1953yv}. The Raychaudhuri equation tells us how the expansion $\theta \equiv \nabla_\mu X^{\mu}$ of a congruence of geodesics $X^\mu$ depends on the shear $\sigma$, vorticity $\omega$, and curvature, and in $d$+0-dimensions is\footnote{Though we will not need these expressions, for context the shear is the traceless symmetric component $\sigma^2 \equiv \sigma_{\mu \nu} \sigma^{\mu \nu} \geq 0$ where $\sigma_{\mu \nu}  \equiv  \frac{1}{2} ( \nabla_\mu X_\nu +\nabla_\nu X_\mu   - \frac{2}{d-1} (\nabla_a X^a) (g_{\mu \nu} - X_\mu X_{\nu}) )$, and the vorticity is the antisymmetric component $\omega^2 \equiv \omega_{\mu \nu} \omega^{\mu \nu} \geq 0$ where $ \omega_{\mu \nu}  \equiv  \frac{1}{2} \left( \nabla_\mu X_\nu - \nabla_\nu X_\mu  \right) $.} 
\begin{equation}
\dot{\theta} = - \frac{1}{d-1} \theta^2 -  \sigma^2 + \omega^2 
- \mathcal{R}_{\mu \nu}X^{\mu} X^{\nu} \, . \label{eq:Raychaudhuri}
\end{equation}
A version of this equation was deployed by Penrose \cite{Penrose:1964wq}, and Penrose and Hawking \cite{Hawking:1969sw}, in  Nobel-prize-winning work showing that the formation of black hole singularities is a generic prediction of general relativity. Penrose \& Hawking's strategy was to use the ($d$+1-dimensional version of the) Raychaudhuri equation to upperbound the expansion of lightsheets, and thereby demonstrate the inevitability of singularities.  
 
Our strategy will be to use the ($d+0$-dimensional version of the) Raychaudhuri equation to upperbound the expansion of geodesic balls.   Step one is to observe that for geodesic balls, the vorticity is zero: one can either argue that geodesic flow conserves angular momentum, or argue that geodesic flow is orthogonal to the constant-time hypersurfaces. Either way, we can put $\omega = 0$. The next step is to drop the shear term $\sigma^2$, turning the equality into an inequality, 
 \begin{equation}
\dot{\theta} +  \frac{1}{d-1} \theta^2  \ \leq \ 
- \mathcal{R}_{\mu \nu}X^{\mu} X^{\nu} \, .\label{eq:preEBGbound}
\end{equation}
(We will have (much) more to say about the shear, and how we can better account for its effect on volume growth, in a subsequent paper \cite{toappear}.) 

To derive the Bishop-Gromov bound, Eq.~\ref{eq:astatementofBGbound},  the next step is to write 
 \begin{equation}
\dot{\theta} +  \frac{1}{d-1} \theta^2  \ \leq \ 
- \mathcal{R}_{\mu \nu}X_\textrm{min}^{\mu} X_\textrm{min}^{\nu} \, . \label{eq:preBGbound}
\end{equation}
The final step is to argue that larger values of $\dot{\theta} +  \frac{1}{d-1} \theta^2$ lead to larger values of the total area and volume: this would be totally trivial if it were just the $\dot{\theta}$ term (bigger expansion gives bigger area), and in Sec.~\ref{sec:monotonicitylemmanegativekappa} we will prove the `monotonicity lemma' that implies it is still true even including the $\frac{1}{d-1} \theta^2$ term. With this lemma, integrating Eq.~\ref{eq:preBGbound} along every geodesic leaving the origin then straightforwardly recovers the Bishop-Gromov bound.

To derive the enhanced-Bishop-Gromov bound, Eq.~\ref{eq:newtheorem}, consider Eq.~\ref{eq:preEBGbound} again. If $\mathcal{R}_{ \mu \nu } X^\mu X^\nu$ were a constant of motion along every geodesic, then the enhanced-Bishop-Gromov bound would follow simply by integrating Eq.~\ref{eq:preEBGbound} along every geodesic leaving the origin. But typically $\mathcal{R}_{ \mu \nu } X^\mu X^\nu$ is not a constant of motion. This means that even if a geodesic starts off with a highly positive value of $\mathcal{R}_{ \mu \nu } X^\mu X^\nu$---which would imply a slow growth rate in that direction---the geodesic can `turn' into a more negative $\mathcal{R}_{ \mu \nu } X^\mu X^\nu$ direction and possibly grow more rapidly. The fact that $\mathcal{R}_{ \mu \nu } X^\mu X^\nu$ may not be conserved along geodesics is thus the major technical obstacle to proving the enhanced-Bishop-Gromov bound.

To deal with this possibility, the major technical insight in this paper will be that geodesic motion is a Hamiltonian flow, and therefore that phase-space volume is conserved. Since phase-space volume is conserved, for every geodesic that turns from a slow direction to a fast direction, there must be another geodesic that turns from a fast direction to a slow direction. By proving, in Sec.~\ref{sec:correlatedJacobis}, a `coefficient shuffling lemma' about the growth rates of correlated families of Jacobi  equations, we will show that the net effect of all this turning is to slow the total rate of growth. In other words, the amount of volume an initially rapidly growing geodesic loses by turning from a fast direction to a slow direction is more than the amount of volume the initially slowly growing geodesic gains in turning from a slow direction to a fast direction. The total volume explored by all the geodesics leaving a point is \emph{reduced} by the fact that $\mathcal{R}_{ \mu \nu } X^\mu X^\nu$ is not conserved, and we can thus upperbound the total rate of growth by pretending there is no turning. This will allow us to prove the enhanced-Bishop-Gromov bound.

\section{Correlated Jacobi equations} \label{sec:correlatedJacobis}

In this section, we will consider properties of solutions $j(t)$ of the Jacobi equation,  
\begin{equation}
j''(t) = \kappa(t) j(t) \ . \label{eq:definitionofjacobi} 
\end{equation}
Unless stated otherwise, we will always take the initial conditions to be $j(0) = 0$ and $j'(0) = 1$. When $j(t)$ is positive,  the rule is that $j(t)$ always evolves according to Eq.~\ref{eq:definitionofjacobi}. But when $\kappa$ is negative, the solution to Eq.~\ref{eq:definitionofjacobi} may itself go negative, and we have a supplemental ad-hoc rule for what happens to $j(t)$ in such a situation, which is that it sticks at zero forever thereafter, i.e.~
\begin{equation}
j(t) = 0 \ \ \ \textrm{for} \ \ t > t_0 \ , \label{eq:adhocrule}
\end{equation}
where $t_0$ is that smallest $t>0$ such that $j(t) = 0$. (This ad-hoc rule is motivated by our eventual application, which will be studying geodesic flows on manifolds; the sign convention for $\kappa$ is chosen for convenience.) Our interest will be in comparing the values of $j(t)$ for different choices of the $\kappa$-schedule $\kappa(\tau)$; since the different $\kappa$-schedules will be correlated we call these `correlated Jacobi equations'.

\subsection{Simple example: two impulses} 
As an illustrative example, let's consider the solution when there are only two, impulsive, contributions to the $\kappa$-schedule, 
\begin{equation}
\kappa(t) = a \, \delta(t-1) + b \, \delta(t-2) . \label{eq:twoimpulses}
\end{equation}
So long as $j(t)$ has never hit zero, the solution is  
\begin{equation}
j(t) = \left\{ \begin{array}{cll}
t & \textrm{for} & 0 \leq t \leq1\\
t + a (t-1) & \textrm{for} & 1\leq t \leq2 \\
t + a (t-1)  + b(2+a)(t-2) & \textrm{for} & 2\leq t 
\end{array} \right. \label{eq:explicitsolutionforj}
\end{equation}
This simple example foreshadows the two results we will prove for general $\kappa(t)$.  
\subsubsection{Foreshadowing monotonicity lemma} 
Our first observation is that larger $a$ and $b$ give rise to larger $j(t)$. In Sec.~\ref{sec:monotonicitylemmanegativekappa} we will generalize this observation to prove the \emph{monotonicity lemma}. This lemma says that if we have two different solutions following two different $\kappa$-schedules, and if  $\kappa_1(t)$ is always bigger than $\kappa_2(t)$,  then  this implies $j_1(t) \geq j_2(t)$. This lemma applies even for negative $\kappa(t)$s. 
\subsubsection{Foreshadowing shuffling lemma} 
Our second observation is that the $ab(t-2)$ term in the last line of Eq.~\ref{eq:explicitsolutionforj} means that the marginal returns of bigger $b$ is bigger when $a$ is bigger. Consider two solution $j_{ab}(t)$ and $j_{AB}(t)$ characterized by two schedules 
\begin{eqnarray}
\kappa_{ab}(t) &\equiv& a \, \delta(t-1) + b \, \delta(t-2) \\ 
\kappa_{AB}(t) &\equiv& A \, \delta(t-1) + B \, \delta(t-2) \ . 
\end{eqnarray}
We can construct another pair of $\kappa$-schedules by ``{shuffling}'' the coefficients
\begin{eqnarray}
\kappa_{aB}(t) &\equiv& a \, \delta(t-1) + B \, \delta(t-2) \\ 
\kappa_{Ab}(t) &\equiv& A \, \delta(t-1) + b \, \delta(t-2) \ . 
\end{eqnarray}
Which grows faster, $j_{AB}(t) + j_{ab}(t)$ or  $j_{Ab}(t) + j_{aB}(t)$? The difference between the sum of the unshuffled or shuffled trajectories is 
\begin{equation}
j_{AB}(t) + j_{ab}(t) - j_{Ab}(t) - j_{aB}(t) = \left\{ \begin{array}{cll}
0 & \textrm{for} & t \leq 2 \\
(A-a)(B-b)(t-2) & \textrm{for} & t \geq 2 
\end{array} \right. \label{eq:explicitdifference}
\end{equation}
This tells us that---given the choice---if we want to make the sum grow as fast as possible, we should pair the coefficients so as to \emph{maximize} the inequality between the two trajectories, pairing the larger value of $\{a,A\}$ with the larger value of $\{b,B\}$. The trajectory that is already growing the fastest makes best use of an additional bigger impulse at $t=2$. In Sec.~\ref{subsec:sortinglemma} we will generalize this observation to prove the \emph{shuffling lemma}. This lemma says that---given the power to instant-by-instant `shuffle' the $\kappa$-schedules amongst the trajectories, so that we can at each instant permute which $j$ is following which $\kappa$---the total growth rate of two trajectories is always maximized by shuffling the $\kappa$-schedules so that one trajectory is always following whichever is the greatest schedule at that instant, and the other is always following whichever is the smallest schedule. Maximizing inequality maximizes total growth. 

     \begin{figure}[htbp] 
    \centering
    \includegraphics[width=6in]{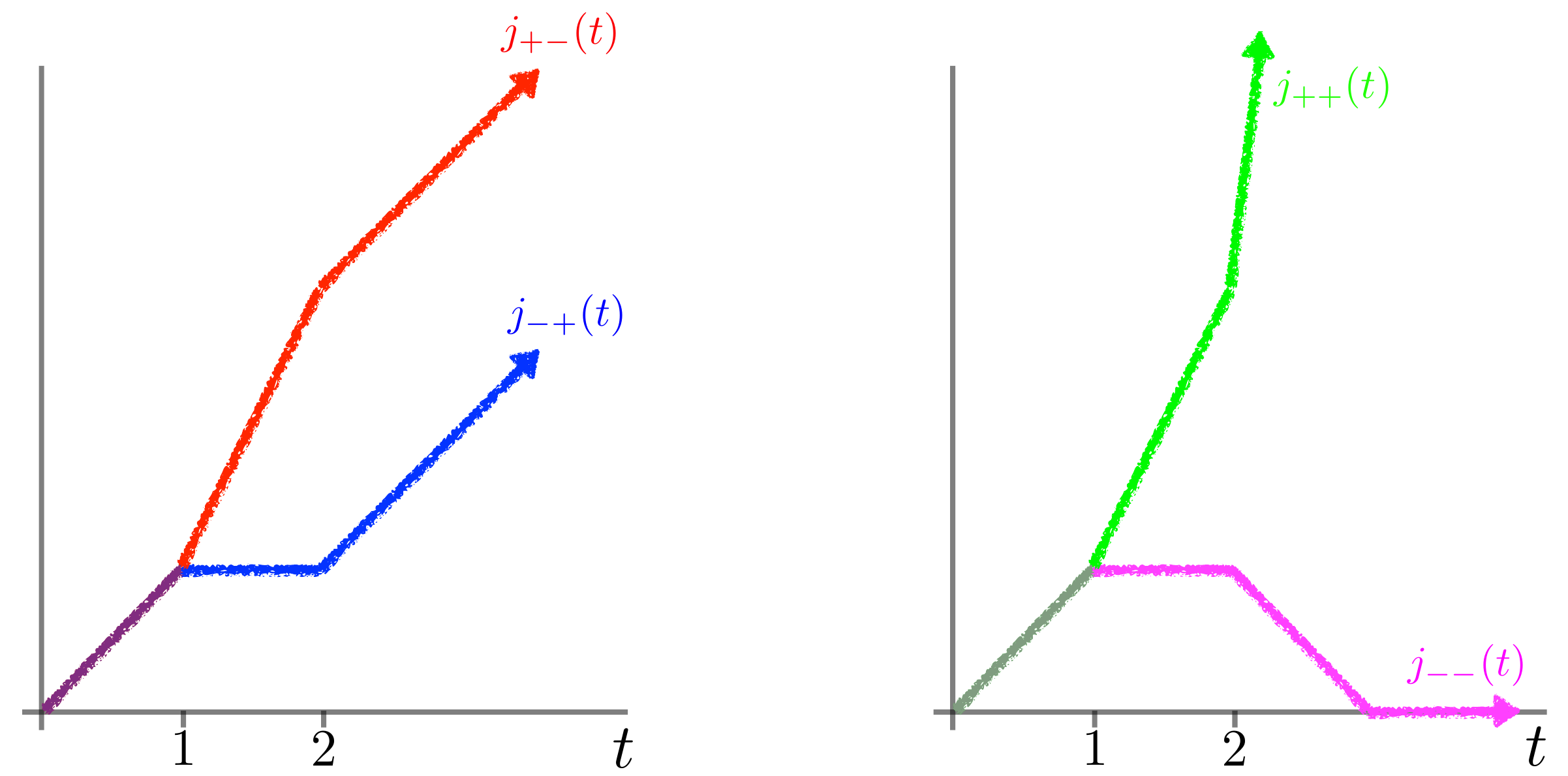} 
    \caption{Left: the solutions to the Jacobi equation with $\kappa$-schedules $\textcolor{red}{\kappa_{+-}(t)}= \delta(t-1) - \delta(t-2) $ and $\textcolor{blue}{\kappa_{-+}(t)} =- \delta(t-1) + \delta(t-2) $. Right: the solutions to the Jacobi equation with  $\textcolor{green}{\kappa_{++}(t)}= \delta(t-1) + \delta(t-2) $ and $\textcolor{magenta}{\kappa_{--}(t) } =-\delta(t-1) - \delta(t-2) $. The monotonicity lemma tells us that $\textcolor{green}{j_{++}(t)} \geq {\textcolor{magenta}{j_{--}(t)}}$ and the shuffling lemma tells us that $\textcolor{green}{j_{++}(t)}  + {\textcolor{magenta}{j_{--}(t)} } \geq {\textcolor{red}{j_{+-}(t)} } + {\textcolor{blue}{j_{-+}(t)} }$.}
    \label{fig:twoimpulses}
 \end{figure}

\subsection{Monotonicity lemma} \label{sec:monotonicitylemmanegativekappa}

Consider two trajectories satisfying $j''_1(t) = \kappa_1(t) j_1(t)$ and $j''_2(t) = \kappa_2(t) j_2(t)$, with initial conditions $j_1(0) = j_2(0) = 0$ \& $j_1'(0) = j_2'(0) = 1$. Let's prove 
\begin{equation}
\textrm{monotonicity lemma}: \ \ \ \  \forall t , \kappa_1(t) \geq \kappa_2(t) \ \ \rightarrow \ \ \forall t, j_1(t)  \geq j_2(t) \ . \ \ \ \ \ \ \ \ \ \ \ 
 \label{eq:monotonicitylemma}
\end{equation}
Note that it is \emph{not} always true that $j'_1(t) \geq j'_2(t)$. For example, consider Eq.~\ref{eq:twoimpulses} with $a_1 = 2, a_2 = 1, b_1 = b_2 = -100$: the impulse at $t=2$ imparts a negative change in velocity to both trajectories proportional to $j(2)$, but since $j_1(2) > j_2(2)$  this gives $j_1$ a more negative velocity. But even though $j_1$ has a more negative velocity, it's not negative enough to catch up with $j_2$ before $j_2$ hits zero, and so doesn't violate the monotonicity lemma.  

Our proof of the monotonicity lemma we will have two steps. The first step will be to show that the quantity $\frac{j'_1}{j_1} - \frac{j'_2}{j_2}$ is nonnegative for all $j_2(t)>0$. This quantity starts off non-negative, since Taylor expanding Eq.~\ref{eq:definitionofjacobi} around $t=0$ gives $\frac{j'_1}{{j_1}} - \frac{j'_2}{{j_2}}=\frac{1}{3} \left( \kappa_1(0) - \kappa_2(0) \right) t + \ldots \geq 0$. We can then argue that the quantity remains non-negative by observing that if it were to cross zero, there would need to be a moment when $\frac{j'_1}{j_1} - \frac{j'_2}{j_2} = 0$ and $\frac{d}{dt}(\frac{j'_1}{j_1} -  \frac{j'_2}{j_2}) < 0$. However this is forbidden because for $\frac{j'_1}{j_1} - \frac{j'_2}{j_2} = 0$ the right hand side of 
\begin{equation}
\frac{d}{dt} \left( \frac{j'_1}{j_1} - \frac{j'_2}{j_2} \right) =  \left( \frac{j''_1}{j_1} - \frac{j''_2}{j_2} \right)  - \left( \frac{(j'_1)^{\, 2}}{j_1^{\, 2}} -\frac{(j'_2)^{\, 2}}{j_2^{\, 2} } \right) =  \kappa_1(t) -  \kappa_2(t)  - \left( \frac{(j'_1)^{\, 2}}{j_1^{\, 2}} -\frac{(j'_2)^{\, 2}}{j_2^{\, 2} } \right), \label{eq:montonicitystep}
\end{equation}
is equal to $\kappa_1(t) - \kappa_2(t)$ and is therefore by assumption nonnegative.

Next, we argue that $j_1(t) \geq j_2(t)$ by observing that if they were to cross, there would need to be a moment when $j_1 = j_2$ and $j'_1 < j'_2$, but the inequality we just established, $\frac{j'_1}{j_1} - \frac{j'_2}{j_2} \geq 0$, forbids this. This establishes Eq.~\ref{eq:monotonicitylemma}. 

For future reference, note that the same reasoning that led to the monotonicity lemma also works if instead of starting at $t=0$ we start at some later time $T$:
\begin{equation}
j_1(T) \geq j_2(T) \ \ \& \  \ \frac{j'_1(T)}{j_1(T)} \geq \frac{j'_2(T)}{j_2(T)} \ \  \& \  \  \forall t \geq T , \kappa_1(t) \geq \kappa_2(t) \ \ \rightarrow \ \ \forall t \geq T, j_1(t)  \geq j_2(t) \ . 
 \label{eq:monotonicitylemmalatestart}
\end{equation}

 \subsection{Shuffling lemma} \label{subsec:sortinglemma}
Consider two trajectories satisfying $j''_1(t) = \kappa_1(t) j_1(t)$ and $j''_2(t) = \kappa_2(t) j_2(t)$. The shuffling lemma says that the total rate of growth  is faster if we shuffle the schedules, 
   \begin{equation}
\textrm{shuffling lemma:} \ \ \ \ \ j_\textrm{max}(t) + j_\textrm{min}(t)  \geq  j_1(t) + j_2(t) \ , \ \ \ \ \ \ \ \ \label{eq:shufflinglemma}
\end{equation}
where $j''_\textrm{max} = \kappa_\textrm{max}(t) j_\textrm{max}(t)$ and $j''_\textrm{min} = \kappa_\textrm{min}(t) j_\textrm{min}(t)$ and we have defined 
\begin{equation}
\kappa_\textrm{max}(t)  \equiv  \textrm{max}[\kappa_1(t),\kappa_2(t)] \ \ \ \& \ \ \  
\kappa_\textrm{min}(t)  \equiv  \textrm{min}[\kappa_1(t),\kappa_2(t)]  \label{eq:definitionofkappaminnoshit} . 
\end{equation}

     \begin{figure}[htbp] 
    \centering
    \includegraphics[width=6in]{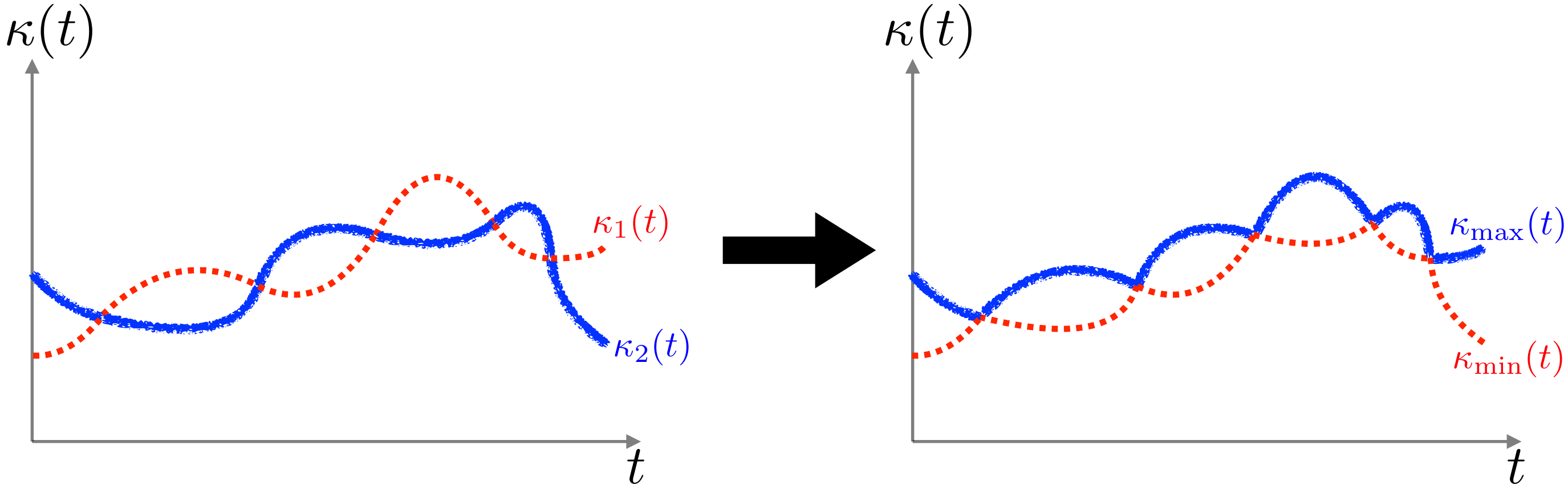} 
    \caption{The shuffling lemma Eq.~\ref{eq:shufflinglemma} says that shuffling from $\kappa_1(t)$ and $\kappa_2(t)$ to ${\kappa}_\textrm{max}(t)$ and ${\kappa}_\textrm{min}(t)$ increases the total rate of growth.}
    \label{fig:switchingkappas}
 \end{figure}
 
\noindent For later reference, note that the monotonicity lemma Eq.~\ref{eq:monotonicitylemma} guarantees that 
\begin{equation}
j_\textrm{max}(t) \geq j_1(t) , j_2(t) \geq j_\textrm{min}(t)  \ , \label{eq:ducksinarow}
\end{equation} 
and therefore as a matter of algebra that $j_\textrm{max} - j_\textrm{min} \geq  0$ and 
\begin{eqnarray}
 \frac{j_\textrm{max} - j_\textrm{min} } {j_\textrm{max} + j_\textrm{min} } & \geq & \frac{|j_1 - j_2|}{j_1+ j_2}  \  \geq \ 0 \ .  \label{eq:actuallnonobviousalgebraresult}
\end{eqnarray}

\subsubsection{Proving shuffling lemma for $j_\textrm{min}(t) > 0$} \label{sec:firstshufflingproof}
First let's prove the shuffling lemma for the time before $j_\textrm{min}(t)$ hits zero. (For some schedules, including schedules with  $\kappa_\textrm{min}(t)$ always positive, $j_\textrm{min}(t)$ will never hit zero, so this era will last forever.) We will prove the lemma by considering the equations of motion for the two quantities
\begin{eqnarray}
\frac{d^2}{dt^2} \left( j_1 + j_2 \right) &=& 
\left(  \frac{\kappa_1 + \kappa_2 }{2}  + \frac{\kappa_1 - \kappa_2 }{2} \frac{ j_1 - j_2 }{j_1 + j_2}  \right)   \left( j_1 + j_2 \right)  \label{eq:kappa1kappa2sum} \\ 
\frac{d^2}{dt^2} \left( j_\textrm{max} + j_\textrm{min} \right) 
& = &\left(  \frac{\kappa_\textrm{max} + \kappa_\textrm{min}}{2} + \frac{\kappa_\textrm{max} - \kappa_\textrm{min} }{2}  \frac{ j_\textrm{max} - j_\textrm{min} }{j_\textrm{max} + j_\textrm{min} }  \right)  \left( j_\textrm{max} + j_\textrm{min} \right)   \ . \label{eq:kappapluskappaminussum}
\end{eqnarray}
These expressions follow as a matter of algebra from Eq.~\ref{eq:definitionofjacobi}. The point is that the $\kappa$-schedule for the second equation of motion is larger than the $\kappa$-schedule for the first, i.e.~that 
\begin{equation}
\frac{\kappa_\textrm{max} + \kappa_\textrm{min}}{2}  + \frac{\kappa_\textrm{max} - \kappa_\textrm{min} }{2}  \frac{ j_\textrm{max} - j_\textrm{min} }{j_\textrm{max} + j_\textrm{min} } \ \geq \ \frac{ \kappa_1 + \kappa_2}{2}  + \frac{\kappa_1 - \kappa_2 }{2} \frac{ j_1 - j_2 }{j_1 + j_2} \ . \label{eq:coefficientisbiggerinthefirst}
\end{equation}
To see this, note that by definition $\kappa_\textrm{max} + \kappa_\textrm{min} = \kappa_1 + \kappa_2$ and $\kappa_\textrm{max} - \kappa_\textrm{min} =  |\kappa_1 - \kappa_2| \geq 0$ and then appeal to Eq.~\ref{eq:actuallnonobviousalgebraresult}. Since the schedule for the quantity $j_\textrm{max}(t) + j_\textrm{min}(t)$ is more positive than the schedule for the quantity $j_1(t) + j_2(t)$, we can use the monotonicity lemma Eq.~\ref{eq:monotonicitylemma} between these two quantities to establish the shuffling lemma, Eq.~\ref{eq:shufflinglemma}. 

\subsubsection{Proving shuffling lemma for $j_\textrm{min}(t) = 0$ \& $j_1(t),j_2(t) > 0$ }

If none of the $j$s ever hit zero, the era described in Sec.~\ref{sec:firstshufflingproof} lasts forever.  If some of the $j$s do hit zero, the monotonicity lemma Eq.~\ref{eq:monotonicitylemma} guarantees that the first to hit zero will be $j_\textrm{min}$.  Let us say $j_\textrm{min}(t)$ hits zero at $t=t_0$. At $t_0$, Eq.~\ref{eq:kappapluskappaminussum} no longer reliably describes the evolution of the quantity $j_\textrm{max}(t) + j_\textrm{min}(t)$, since our rule Eq.~\ref{eq:adhocrule} is that $j_\textrm{min}(t) = 0$ for all $t \geq t_0$. In Sec.~\ref{sec:firstshufflingproof} we proved that for $t<t_0$ we have $j_\textrm{max}(t) + j_\textrm{min}(t) \geq j_1(t) + j_2(t)$ and $\frac{j'_\textrm{max}(t) + j'_\textrm{min}(t)}{j_\textrm{max}(t) + j_\textrm{min}(t) } \geq \frac{j'_1(t) + j'_2(t)}{j_1(t) + j_2(t)}$. Since $j_\textrm{min}(t_0) = 0$ and $j_\textrm{min}'(t_0) \leq 0$, this means that immediately following $t_0$ we have 
\begin{equation}
j_\textrm{max}(t) \geq j_1(t) + j_2(t)  \ \ \ \textrm{ and } \ \ \ \frac{j'_\textrm{max}(t) }{j_\textrm{max}(t) } \geq \frac{j'_1(t) + j'_2(t)}{j_1(t) + j_2(t)} \ . 
\end{equation}
Indeed, these two inequalities remain true throughout this era. This follows from applying the monotonicity lemma Eq.~\ref{eq:monotonicitylemmalatestart} between the quantity $j_\textrm{max}(t)$ and the quantity $j_1(t) + j_2(t)$,  since the equation of motion for $j_1(t) + j_2(t)$ may be written 
\begin{equation}
(j_1 + j_2)'' = \left( \kappa_\textrm{max}  -  \frac{( \kappa_\textrm{max} - \kappa_1) j_1 }{j_1 + j_2}     -  \frac{ ( \kappa_\textrm{max} - \kappa_2) j_2}{j_1+ j_2}    \right) (j_1 + j_2)  \leq \kappa_\textrm{max} (j_1 + j_2) \ . 
\end{equation}
This establishes the shuffling lemma in this era.

\subsubsection{Proving shuffling lemma for $j_\textrm{min}(t) = j_1(t)=  0$ \& $j_2(t) > 0$}
 The next era begins when the first of $j_1(t)$ or $j_2(t)$ hits zero; without loss of generality let's say $j_1(t) = 0$. In this era, the shuffling lemma is just the statement that $j_\textrm{max}(t) \geq j_2(t)$, which follows directly from the monotonicity lemma between $j_\textrm{max}(t)$ and $j_2(t)$.

\subsubsection{Proving shuffling lemma for $j_\textrm{min}(t) = j_1(t)=  j_2(t) = 0$}
 Finally, in the era when $j_\textrm{min}(t)$, $j_1(t)$, and $j_2(t)$ have all hit zero the shuffling lemma says $j_\textrm{max}(t) \geq 0$, which is trivially true.  
This completes our proof  for all eras.

\subsection{Shuffling lemma for higher powers} \label{subsec:sortinglemmahigherpowers}
The shuffling lemma Eq.~\ref{eq:shufflinglemma} remains true if we raise each of the quantities to a higher power 
\begin{equation}
\textrm{higher-power shuffling lemma:} \ \ \ \ \ \ \ \   p \geq 1 \ \ \rightarrow \ \  \ \ j_\textrm{max}(t)^p + j_\textrm{min}(t)^p  \geq  j_1(t)^p + j_2(t)^p \ . \ \ \ \ \ \ \ \ \ 
 \label{eq:generalizedswitching}
\end{equation}
 We will argue by induction for integer $p$ (which is all we need in this paper). First recall that Eq.~\ref{eq:generalizedswitching} is true for $p=1$, as established by Eq.~\ref{eq:shufflinglemma}. Now assume that it is true for all integer values smaller than $p$, we'll show that it must also be true for $p$. Since it is true for all values up to $p-1$ we have  
\begin{equation}
(  j_\textrm{max}(t) + j_\textrm{min}(t))(  j_\textrm{max}(t)^{p-1} + j_\textrm{min}(t)^{p-1})  \geq (  j_1(t) + j_2(t))(  j_1(t)^{p-1} + j_2(t)^{p-1})  \ . \label{eq:plusdoubledup}
\end{equation}
On the other hand Eq.~\ref{eq:ducksinarow} tells us that 
\begin{equation}
(  j_\textrm{max}(t) - j_\textrm{min}(t))(  j_\textrm{max}(t)^{p-1} - j_\textrm{min}(t)^{p-1})  \geq (  j_1(t) - j_2(t))(  j_1(t)^{p-1} - j_2(t)^{p-1})  \ . \label{eq:minusdoubledup}
\end{equation}
 Adding together Eqs.~\ref{eq:plusdoubledup} and \ref{eq:minusdoubledup} gives Eq.~\ref{eq:generalizedswitching}, completing the proof by induction.

\subsection{Shuffling lemma for many trajectories} \label{subsec:sortinglemmalotsofkappas}

Let's generalize the shuffling lemma to more than two solutions. We will prove that for any set of solutions $j''_i(t) =  \kappa_i(t) j_i(t)$, if granted the power at each moment to permute which trajectory $j_i$ is following which schedule $\kappa_i$,  the largest value of
\begin{equation}
\sum_i j_i(t)^p ,
\end{equation}
for all $p \geq 1$, is achieved by perfectly ordering the schedules, so that one trajectory always follows whatever is the largest $\kappa_i(t)$ at that time, another trajectory always follows the second largest, another trajectory always follows the third largest, etc., as in Fig.~\ref{fig:orderinglemma}. To maximize the total value of all the trajectories, we should maximize the inequality between the trajectories. This follows directly from the two-schedule result, Eq.~\ref{eq:generalizedswitching}, together with the fact we can order any number of schedules by repeatedly iterating pairwise shufflings.

     \begin{figure}[htbp] 
    \centering
    \includegraphics[width=6in]{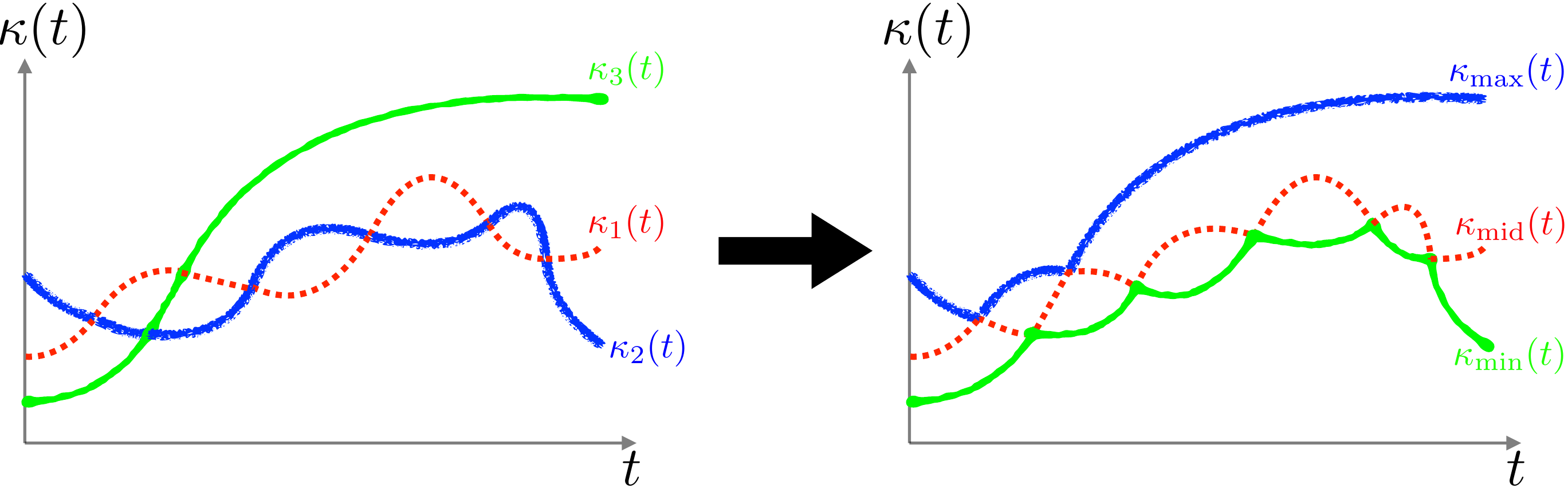} 
    \caption{Shuffling from $\kappa_1(t)$ and $\kappa_2(t)$ and $\kappa_3(t)$ to ${\kappa}_\textrm{max}(t)$ and  ${\kappa}_\textrm{mid}(t)$ and ${\kappa}_\textrm{min}(t)$ increases the total rate of growth, $j_\textrm{max}(t) + j_\textrm{mid}(t) + j_\textrm{min}(t) \ \geq \ j_1(t) + j_2(t) + j_3(t)$. This follows from the two-schedule shuffling lemma, depicted in Fig.~\ref{fig:switchingkappas}, since we can perfectly order any number of trajectories using pairwise shuffles.} 
    \label{fig:orderinglemma}
 \end{figure}

\subsection{The continuum limit of shuffling} 

In the proofs above, we showed that shuffling so as to perfectly order the schedules always increases the total rate of growth. As presented, our proof only established this for a \emph{finite} number of shuffles between a \emph{finite} number of trajectories. In our intended application, we will actually be interested in applying the lemma for an uncountable number of shuffles (one shuffle at each instant of time) between an uncountable number of trajectories (one trajectory per geodesic leaving each point). From a physics point of view, it is clear this will not be a problem: the Jacobi equation is a second-order equation, so the trajectories have inertia, which smooths the UV; and in any event the $\kappa$-schedules $\mathcal{R}_{\mu \nu}X^{\mu} X^{\nu}$ will be differentiable along geodesics. This means continuous shuffling can be arbitrarily well approximated by a large but finite number of shuffles. Nevertheless, let's argue that more carefully. 

An important special case for the switching lemma is when $\kappa(t)$ is the sum of Dirac measures. In our applications to differential geometry we will have to switch uncountably many solutions uncountably many times, so measure theory is a handy tool. We will pass from discrete probability measure spaces and $\kappa(t)$ atomic to general probability measure spaces and general (e.g.\ continuous) $\kappa(t)$. The discrete probability measure space $\{\tau_1, \dots, \tau_n\} \coloneqq T_n$ will map to and approximate the general probability measure space $T$.

Let $\mathcal{F}_0$ denote a family of (uncoupled) Jacobi equations
\begin{equation}\label{eq:jacfam}
	\mathcal{F}_0 = \{j_\tau^{\pr \pr}(t) = \kappa_\tau(t) j_\tau(t), \ 0 \leq t \leq t_0, \text{ with standard initial conditions}\}
\end{equation}
with the monotonicity property that if $\kappa_\tau(t) > \kappa_{\tau^\pr}(t)$ for some $t \in [0, t_0]$ then $\kappa_\tau(t) \geq \kappa_{\tau^\pr}(t)$ for all $t \in [0, t_0]$. The parameter $\tau \in T$, with $T$ being a probability measure space with measure $\mu_T$. The coefficient function $\kappa_\tau(t) \coloneqq \kappa(\tau, t)$ may be a continuous function $T \times [0,t_0] \ra \R$ to the Reals. But, more generally, $\kappa(\tau,t)$ can also be a weak-$\ast$ limit of continuous functions, e.g. $\kappa(\tau,t)$ may be a measurable function or merely a measure. We assume that each Jacobi equation in $\mathcal{F}_0$ has a solution on $[0, t_0]$. There is an important special case of (\ref{eq:jacfam}), which finds applications in Sec.~\ref{sec3}, and that is the case where all coefficient functions $\kappa_\tau(t)$ are constant w.r.t.\ $t$:
\begin{equation}\label{eq:cons}
	\text{constant case: } \mathcal{F}_0 = \{j_\tau^{\pr \pr}(t) = \kappa_\tau j(t) , \ 0 \leq t \leq t_0, \text{ with standard initial conditions}\}  . 
\end{equation}

From initial families (\ref{eq:jacfam}) or (\ref{eq:cons}) many other families $\mathcal{F}_\sigma$ can be made by a procedure we continue to call \emph{coefficient shuffling}, it being a limit of the discrete \emph{coefficient shuffling} explained above:
\begin{equation}
	\mathcal{F}_\sigma = \{j_\tau^{\pr \pr} = \kappa_{\sigma(t)\tau} j(t) , \ 0 \leq t \leq t_0, \ \text{with standard initial conditions}\} , 
\end{equation}
where $\sigma(t)$ is a 1-parameter family of measure-automorphisms of $T$. That is, $\sigma: T \times [0, t_0] \ra T \times [0, t_0]$ is a measurable function such that for all $t \in [0, t_0]$, $\sigma_t \coloneqq \sigma\vert_{T \times t}: T \times t \ra T \times t$ preserves levels, and satisfies, for all measurable $S \subset T$, $\mu_T(\sigma_t(S \times t)) = \mu_T (S \times t)$. We also assume $\sigma(0) = \id_{T \times 0}$. Observe that in the case $\sigma = \id_{T \times [0, t_0]}$, $\mathcal{F}_\sigma = \mathcal{F}_0$.

For each $s \in [0, t_0]$ and ``power'' $p$, $1 \leq p < \infty$, define the \emph{total solution}
\begin{equation}
	\mathrm{Tot}(s,p,\mathcal{F}_\sigma) = \int_T (\mathrm{Sol}(\mathcal{F}_\sigma)\Big\vert_s)^p , 
\end{equation}
where $\mathrm{Sol}(\mathcal{F}_\sigma)$ denotes the set of solutions to the standard-initial-conditions Jacobi equations in the family $\mathcal{F}_\sigma$. The symbol $\vert_s$ directs us to evaluate the equations at $s \in [0, t_0]$ and then we integrate the $p^{\text{th}}$ power (which is positive since we do not integrate beyond focal points) of the evaluations over $T$.

\begin{thm}
	Assume $\mathcal{F}_0$ is a monotone family ($\kappa_\tau(s)>\kappa_\tau'(s)$, for some $s$ in $[0,t]$ implies $\kappa_\tau(s)\geq \kappa_\tau'(s)$ for all $s$ in $[0,t]$), then for all $s \in [0, t]$ and $p \geq 1$,
	\[
	\mathrm{Tot}(s,p,\mathcal{F}_0) \geq \mathrm{Tot}(s,p,\mathcal{F}_\sigma) \ . 
	\]
\end{thm}

\begin{cor}[Constant is best]
	Theorem 1 holds, of course, in the constant coefficient case where line (\ref{eq:jacfam}) is specialized to (\ref{eq:cons}). It says the integrals are maximized by leaving the coefficients constant and \emph{not} shuffling.
	\qed
\end{cor}

\begin{proof}[Proof of Theorem 1]
	Every Borel measure $\mu$ is the limit in the weak-$\ast$ topology of finite Dirac measures $\mu_n =  \sum {c}\, \delta(x - x_n)$. That is, for any continuous test function $h$
	\begin{equation}\label{eq:conth}
		\int_{\mu_n} h \ra \int_\mu h \ . 
	\end{equation}
	To see this, note that every continuous function is detected by some finite combinations of Dirac measures, so the measures are dense in the dual, $\text{(cont. func.)}^\ast = \{\text{Borel measures}\}$. Since the arguments in this section concerns coefficients $\kappa(t)$ which are merely measurable functions, we record in Lemmas 1 and 2 some basic analytic properties of Jacobi solutions so far treated implicitly. 
	\begin{lemma}[Existence and Uniqueness]
		For general initial conditions and any continuous function $\kappa(t)$, or the weak-$\ast$ limit of these in the distributional sense (these includes measurable $\kappa(t)$ and Dirac-like measures), there exists a $t_0 > 0$ so that the solution $j(t)$ is defined, continuous, and unique on $[0, t_0]$ with a distributional first derivative. If $\kappa_0(t)$ is continuous and bounded above, any $t_0 \in \R^+$ suffices. In this case, $j(t)$ is $C^1$ with a continuous second derivative.
	\qed
	\end{lemma}
\noindent We always assume $\kappa(t)$ is continuous or a weak-$\ast$ limit of continuous functions.

	\begin{lemma}[Continuity]\label{lm:continuity}
		For all $t \in (0,t_0]$, $j(t)$ varies continuously with initial conditions, and as the coefficients $\kappa(t)$ are varied in the weak-$\ast$ topology.
		\qed
	\end{lemma}

	Line (\ref{eq:conth}) allows us to approximate any probability measure space (pms) $T$ by a finite pms. The time parameter $[0, t_0]$ can also be discretized. By Lemma \ref{lm:continuity}, if the  conclusion of Theorem 1 failed, the counterexample would survive a finite approximation. Concretely, $\mu_n$ is the uniform probability measure space on $\{1, \dots, n\}$, and the approximation is a mapping $\tau$ from $\mu_n$ to $T$; denote 
 $\tau(i)$ by $\tau_i$, $1 \leq i \leq n$. Next we replace $[0, t_0]$ by a discrete approximation $\{p_0, \dots, p_N\}$, where $p_j = t_0/j$. The function $\kappa(\tau, t)$ is now approximated by discrete atomic coefficients
	\begin{equation}
		\overline{\kappa}(\tau, t): \{1, \dots, n\} \times \{p_0, \dots, p_N\} \ra \R 
	\end{equation}
	and $\sigma: T \times [0, t_0] \ra T \times [0, t_0]$ is approximated by $N-1$ permutations of $\{1, \dots, n\}$,
	\begin{equation}
		\overline{\sigma}: \{1, \dots, n\} \times p_j \ra \{1, \dots, n\} \times p_j,\ 1 \leq j \leq N-1 . 
	\end{equation}

	Concretely, $\sigma$ is being approximated by the step function $\lbar{\sigma}_{\textrm{step}}$, constant on the intervals $(p_i, p_{i+1}]$, $0 \leq i \leq N$. The approximation is in the sense that for any bounded continuous test function $f: T \times [0, t_0] \ra \R$
	\begin{equation}
		\int  f(\sigma(\tau, t)) - \int f(\bar{\sigma}_{\text{step}}(\{\tau_1,\dots,\tau_n\}, t)) \ra 0 \text{ as } (n,N) \ra (\infty,\infty)
	\end{equation}
	where the first integral is w.r.t. $T \times $Lebesgue measure on $[0,t_0]$, and the second integral is w.r.t. $\mu_n \times$Lebesgue measure on $[0,t_0]$. 

Thus the finite Shuffling Lemma (\ref{eq:shufflinglemma}) passes through the limits, yielding Theorem 1.
\end{proof}

\section{Geometry: the enhanced-BG theorem} \label{sec3}
In this section, we  consider the rate of growth of geodesic balls, and prove various enhancements of the Bishop-Gromov theorem. In Sec.~\ref{sec:BGfinite}, we upperbound the average rate of growth of geodesic balls, averaged over all starting points, for finite-volume inhomogeneous Riemannian manifolds. In Sec.~\ref{sec:BGinfinite} we upperbound the same quantity for infinite-volume inhomogeneous Riemannian manifolds; in this case our bounds are difficult to evaluate. In Sec.~\ref{sec:BGhomogeneous} we upperbound the rate of growth of geodesic balls for  homogeneous Riemannian manifolds, of either finite or infinite volume, and show that in this case the bound significantly simplifies. 

\subsection{Enhanced-BG for finite volume spaces} \label{sec:BGfinite}

Let $T$ be the unit tangent bundle to $M$ with its induced Riemannian metric, and associated volume form, the ``Liouville measure'' $\mu_T$. Geodesic flow: $G(t): T \ra T$ preserves this measure,~i.e. for $S \subset T$ measurable, $\mu_T(G(t)(S)) = \mu_T(S)$ for all $t$.

We adopt a symmetrical perspective and consider volume growth from \emph{all} points $q \in M$ at once. The most natural integral to write down is:
\begin{equation}\label{eq:area}
	\text{Total area } = \textrm{TA}(t) = \int_{\tau \in T} \det(J_\tau(t)),
\end{equation}
where the integration measure is $\mu_T$, and will always be implicit when we integrate over $T$.

Above, $J$ is the Jacobi tensor field normal to $\gamma$, the geodesic on $M$ with initial condition $\tau \in T$. $J$ is taken to have standard initial conditions $J(0) = 0$ and $J^\pr(0) =\id$ on the orthonormal $(d-1)$-space to $\gamma_\tau$ at $\gamma_\tau(0)$. $J$ may be defined (see \cite{ballman}) by the condition that it obeys an operator Jacobi equation:
\begin{equation}\label{eq:opeq}
	J^{\pr \pr}(t) = R_\gamma(t) J(t) , 
\end{equation}
where $R_\gamma(t)$ is the symmetric tensor defined by
\begin{equation}
	R_\gamma X = R(X, \dot{\gamma})(t)\dot{\gamma}(t), \ R \text{ the Riemann tensor,}
\end{equation}
and $\dot{\gamma}$ is short hand for $\frac{d \gamma(t)}{dt}$. In other words, $R_\gamma(t)$ is the component of the Ricci quadratic form in the direction of the geodesic $\gamma$; in the notation of Sec.~\ref{sec:heuristicphysicsproof} this is $\mathcal{R}_{\mu \nu}X^\mu X^\nu$. 

By considering nearby geodesics $\gamma_\tau$, $\tau = (q, v) \in T$ emanating from the same point $q \in M$, we obtain the Weingarten operator $\U$ acting orthogonally to $\gamma$. Let $\rho$ be the distance function from $q$, then:
\begin{equation}
	\U X \coloneqq \nabla_X\dot{\gamma} = \nabla_X \operatorname{grad} \rho,\ \text{for } X \text{ orthogonal to } \gamma \ . 
\end{equation}

It should be noted that the operator $\U$ contains the same information as the second fundamental form on the sphere of radius $\rho$ about $q$. So $\U J = \nabla_J \operatorname{grad} \rho = \nabla_s \de_t \gamma(s,t)\vert_{s=0} = \nabla_t \de_s \gamma(s,t) \vert_{s=0} = \frac{\de}{\de t}J$, the second to last equality is by exchanging the order of differentiation, parallel to $\gamma$ and transverse to $\gamma$, which is allowed since the parameters $s$ and $t$ are involutory; together they define a smooth 2D sheet. The result is
\begin{equation}
	\U = J^\pr J^{-1}
\end{equation}
and obeys the operator Riccati equation
\begin{equation}\label{eq:opric}
	\U^\pr + \U^2 + R_\gamma = 0
\end{equation}
as long as $J(t)$ remains invertible, as can be seen by expanding
\begin{equation}\label{eq:expand}
	\U^\pr J = (\U J)^\pr - \U J^\pr = J^{\pr \pr} - \U^2 J = -R_\gamma J - \U^2 J \ . 
\end{equation}

In this paper we only consider Jacobi equations along geodesics $\gamma$ up to their first conjugate point, so that $\det(J) > 0$. Thus we may cancel $J$ from line (\ref{eq:expand}). Now define
\begin{equation}\label{eq:plusplus}
	u := \frac{1}{d-1} \frac{d}{dt} \log(\det(J)) = \frac{1}{d-1} \mathrm{tr}( {J}\vphantom{J}^\pr  {J}\vphantom{J}^{-1}) = \frac{1}{d-1} \mathrm{tr}(\U) \ . 
\end{equation}
It follows that 
\begin{eqnarray}
	u^\pr & = & \frac{1}{d-1}(\mathrm{tr}(\U^\pr))  \\
	& = &- \frac{1}{d-1} \mathrm{tr}(\U^2) - \frac{1}{d-1} \mathrm{tr}(R_\gamma) \  \ \ \ \text{ (from line (\ref{eq:opric}))}  \label{eq:presplits} \\
	& \leq & - \frac{1}{(d-1)^2} (\mathrm{tr}(\U))^2 - \frac{1}{d-1}\mathrm{Ric}(\dot{\gamma}, \dot{\gamma}) , \label{eq:splits}
\end{eqnarray}
where to move from Eq.~\ref{eq:presplits} to Eq.~\ref{eq:splits} apply Cauchy-Schwarz to the first term, and for the second term note that $\operatorname{tr}(R_\gamma) = \sum_i \langle R(e_i,\dot{\gamma})\dot{\gamma},e_i \rangle$. (In the language of the Raychaudhuri equation, Eq.~\ref{eq:Raychaudhuri}, $u$ is the expansion $\theta$, and moving from Eq.~\ref{eq:presplits} to Eq.~\ref{eq:splits} is equivalent to dropping the shear.) With these manipulations, Eq.~\ref{eq:splits} becomes
\begin{equation}
	u' \leq  -u^2 - \frac{1}{d-1} \mathrm{Ric}(\dot{\gamma},\dot{\gamma}) . 
\end{equation}
(This inequality is equivalent to Eq.~\ref{eq:preEBGbound}.) In order to transform this into the form a Jacobi equation, and thus to make contact with the results of the last section, it will be helpful to define
\begin{equation}\label{eq:thus}
	\kappa \coloneqq u^\pr +  u^2 \leq -  \frac{1}{d-1} \mathrm{Ric}(\dot{\gamma},\dot{\gamma}) \ . 
\end{equation}
By setting $\det( {J}) \coloneqq j^{d-1}$ and plugging into Eq.~\ref{eq:plusplus}, we see that $j$ satisfies the usual relation between $u$ and $j$, solutions to the scalar Riccati and Jacobi equations respectively, 
\begin{equation}
	u = \frac{j^\pr}{j} \ . 
\end{equation}

This relationship means that $j$ satisfies the scalar Jacobi equation with the same coefficient $\kappa(t) \coloneqq -u^\pr(t) - u^2(t)$ appearing in the Riccati equation solved by $u$. Thus line \ref{eq:area} can be extended to
\begin{equation}
	\operatorname{TA}^+(t) = \int_T \det( {J}_\tau(t)) \leq \int_T (j_\tau(t))^{d-1} \leq \int_T (j_\tau^{\text{Ric}}(t))^{d-1}
\end{equation}
where $j_\tau$, with standard initial conditions, solves the Jacobi equation with $\kappa(t) = -u^\pr - u^2$ and $j_\tau^\text{Ric}$ solves the Jacobi equations, with standard initial conditions, for $\kappa(t) = \frac{1}{d-1}\mathrm{Ric}(\dot{\gamma}(t), \dot{\gamma}(t))$. The last equality follows from Eq.~\ref{eq:thus}. 

Now apply Theorem 1 to the ($\tau \in T$)-family $\mathcal{F} = \{j_\tau^{\text{Ric}}, \tau \in T\}$. This is \emph{not} generally a constant coefficient family but the invariance of Liouville measure under a geodesic flow says that it \emph{comes} from some constant coefficient family by ``shuffling coefficients'' according to some time-dependent measure automorphism $\sigma(t)$. Thus $\mathcal{F} = \mathcal{F}_{\sigma}$ coming from the appropriate constant coefficient family $\mathcal{F}_0$. Thus Theorem 1 implies
\begin{equation}
	\operatorname{TA}(t)  \leq \int_T (j_\tau^{\text{Ric}}(t))^{d-1} \leq \int_T (j_\tau^{\text{constant}}(t))^{d-1} \label{eq:thisequationhere123}
\end{equation}
where $\{j_\tau^{\text{constant}}\}$ obeys the constant coefficient Jacobi equations with standard initial conditions where the coefficients are distributed according to the $\mu_T$ density of the function $\frac{1}{d-1}\mathrm{Ric}(\dot{\gamma}(t),\dot{\gamma}(t))$ on $(T, \mu_T)$. But $j_\tau^{\text{constant}}$ may be solved explicitly as $j_\tau^{\text{constant}} = \textrm{sn}_k(t)$ where, as in Eq.~\ref{eq:definitionofsnfunction}, 
\begin{equation}
\operatorname{sn}_k(t)  \equiv \operatorname{sn}(k,t)  \equiv 
\left\{ \begin{array}{ccccl} 	
\frac{\sin(\sqrt{k}t)}{\sqrt{k}} & \textrm{for} & k>0& \& & 0 \leq t \leq \frac{\pi}{\sqrt{k}}\\
0 & \textrm{for} & k>0& \& & \ \, \ \ \  \  t \geq  \frac{\pi}{\sqrt{k}}\\
t & \textrm{for} & k=0& \& & \ \ \ \ \ \, t \geq 0\\
 \frac{\sinh(\sqrt{-k}t)}{\sqrt{-k}} & \textrm{for} & k<0& \& & \ \ \ \ \ \, t \geq 0\\
 \end{array} \right.   \label{eq:definitionofsnfunction2}
\end{equation}
This gives:
\begin{equation}\label{eq:taeq}
	\operatorname{TA}(t) \leq \int_{\mu_T} (\textrm{sn}_k(t))^{d-1} \ . 
\end{equation}

Now, $\frac{\operatorname{TA}(t)}{\operatorname{Vol}(M)}$ has the interpretation as the average growth rate of volume in $M$ as balls are expanded at speed $=1$ across radius $t$ starting from all $q \in M$. So the rhs of line (\ref{eq:taeq}) yields an upper bound on this growth rate. That upper bound can further be integrated over $t$ to give an upper bound on the average volume encountered by radius $t$. Since the function on the rhs of (\ref{eq:taeq}) is familiar from the constant curvature geometry, we state Theorem 2 in this context.\footnote{While there is a simple closed formula for the area of a radius $t$ sphere in hyperbolic $n$-space $\mathbb{H}^n$ as a function of $t$ and $n$, there is no comparable formula for  the volumes of balls, so we treat this quantity as a primitive.}

\begin{thm}
Up to the radius $t_0$ at which the nearest focal point is encountered, the average area $\operatorname{AA}[M](t)$ of the sphere of radius $t$ inside $M$, $\operatorname{AA}[M](t) \coloneqq \frac{\operatorname{TA}[M](t)}{\operatorname{Vol}(M)}$ satisfies
	\[
	\operatorname{AA}[M](t) \leq \int_T ({sn}_k(t))^{d-1}/\mathrm{Vol}(M) , 
	\]
	and the average volume $\operatorname{AV}[M](t)$ of a ball of radius $t$ in $M$ satisfies
	\[
	\operatorname{AV}[M](t) \leq \int_T C_k^d(t)/\mathrm{Vol}(M) , 
	\]
	where:
	\begin{align*}
		& C_k^d(t) = \frac{\textup{Vol}\left(\textup{radius }t\textup{-ball in }H_k^d\right)}{\Omega_{d-1}} \ .
	\end{align*}
	Here $\Omega_{d-1}$ is the area of a unit sphere in $R^d$ and $H_k^d$ is the simply connected maximally symmetric space of dimension $d$ and scaled so that all sectional curvatures equal $k$.
	\qed
\end{thm}

Let us reproduce Theorems 5.1 and 5.3 from \cite{ballman}, sometimes called Bishop-Gromov inequalities for area and volume respectively. The context is a complete Riemannian manifold $M$, a point $q \in M$ and the various unit speed geodesics $\gamma_\tau: [0,b) \ra M$ emanating from $q$ with initial condition $\tau = (q,v)$. As before, $J$ is the Jacobian tensor field with standard initial conditions $J(0) = 0$, $J^\pr(0) = \id$. If $R_\gamma$ is constant, $R_\gamma = k I_\gamma$, $k \in \R$ and $I_\gamma$ the field of identity operators along $\gamma$, then these constant coefficient solutions determine the denominators on the right-hand sides below.
\begin{customthm}{(5.1 in \cite{ballman})}
	Assume that $\operatorname{Ric}(\dot{\gamma},\dot{\gamma}) \geq (d-1)k$, $k \in \R$, and that $\det(J) > 0$ on $(0,b)$. Then:
	\[
		1 \geq \frac{\det(J(r))}{\operatorname{sn}_k^{d-1}(r)} \geq \frac{\det(J(s))}{\operatorname{sn}_k^{d-1}(s)}
	\]
	for all $r < s$ in $(0,b)$. The left inequality is strict unless $R_\gamma = k I_\gamma$ on $[0,r]$. The right inequality is strict unless $R_\gamma = k I_\gamma$ on $[0,s]$.
\end{customthm}

\begin{customthm}{(5.3 in  \cite{ballman})}
	Assume that $\operatorname{Ric}(\dot{\gamma},\dot{\gamma}) \geq (d-1)k$ for all unit speed geodesics $\gamma_\tau$, some fixed $k \in \R$. Then
	\[
		1 \geq \frac{V_q(r)}{V_k(r)} \geq \frac{V_q(s)}{V_k(s)}
	\]
	for all $0 < r < s \leq \max_{p}\{\operatorname{dist}(p,q)\}$ (with $s < \frac{\pi}{\sqrt{k}}$ if $k > 0$). The left inequality is strict unless $B_q(r)$ is isometric to $B_{q,k}(r)$ and the right inequality is strict unless $B_q(s)$ is isometric to $B_{q,k}(s)$. $B_{q,k}(r)$ refers to the ball of radius $r$ in the simply connected symmetric space with constant sectional curvature $= k$.
\end{customthm}

\noindent Applying coefficient shuffling we obtain the following theorems bounding certain averages of  area and volume, respectively.

\begin{thm}\label{thm:com1}
	Let $M$ be a complete finite volume Riemannian manifold and let $\Delta$ be the probability measure (pm) on $\R$ giving the density of unit tangent vectors $\tau$ with any fixed value $\in \R$ of $\operatorname{Ric}(\tau,\tau)$. Now using our previous notations
	\[
		1 \geq \frac{\int_{\tau \in T} \det( {J}(r))}{\int_{k \in \R} \operatorname{sn}_k^{d-1}(r)} \geq \frac{\int_{\tau \in T} \det( {J}(s))}{\int_{k \in \R} \operatorname{sn}_k^{d-1}(s)}
	\]
	where the integrals in the denominators are w.r.t\ $\Delta$, and $0 < r < s < \infty$. Equality holds exactly under the conditions where it held in Thm 5.1. \cite{ballman}. \qed
\end{thm}

\begin{thm}\label{thm:com2}
	Let $M$ be a complete finite volume Riemannian manifold and $\Delta$ be the pm of Theorem \ref{thm:com1}, then:
	\[
		1 \geq \frac{\int_{(q,v) = \tau \in T} \int_{S_q^{d-1}} \ d\Omega \int_{t=0}^r \ dt \ \det( {J}(t))}{\int_{k \in \R} \Delta \int _{t=0}^r \ dt \ \operatorname{sn}_k^{d-1}(t)} \geq \frac{\int_{\tau \in T} \int_{S_q^{d-1}} \ d\Omega \int_{t=0}^s \ dt \ \det( {J}(t))}{\int_{k \in \R} \Delta \int_{t=0}^s \ dt \ \operatorname{sn}_k^{d-1}(t)}
	\]
	for $0 < r < s < \infty$, where $d\Omega$ is the measure on the tangential unit sphere. By convention, the inner integrations are terminated where Jacobi fields focus. The outer denominator integrals are w.r.t\ the pm $\Delta$. Equalities hold exactly under the conditions stated in Thm 5.3 \cite{ballman}. \qed
\end{thm}

The proofs are a routine addition of coefficient shuffling to the proofs in \cite{ballman}, and thus are not given. The ingredients are our line \ref{eq:thus}, an initial estimate (line 30 \cite{ballman})
\begin{equation}
	\lim_{r \ra 0} \frac{\det(J(r))}{r^{d-1}} = 1
\end{equation}
and a Sturm comparison, stated here for the reader's convenience:

\begin{customlm}{(4.1 in \cite{ballman})}
	Let $u,v: (a,b] \ra \R$ be smooth with $u^\pr + u^2 \leq v^\pr + v^2$ and assume that $u^\pr + u^2$ and $v^\pr + v^2$ extend smoothly to $[a,b]$. Then the limits $u(a) = \lim_{t \ra 0} u(t)$ and $v(a) = \lim_{t \ra 0} v(t)$ exist as extended real numbers in $(-\infty, \infty]$. If $u(a) \leq v(a)$ then $u \leq v$ on $(a,b]$ with equality $u = v$ iff $u(b) = v(b)$.
\end{customlm}

\subsection{Enhanced-BG for infinite volume spaces } \label{sec:BGinfinite} 

We will now consider what bounds we can place on the growth of geodesic balls for inhomogeneous spaces of infinite volume. As we will see, in this case progress will be limited. Unlike in finite-volume inhomogeneous spaces, considered in Sec.~\ref{sec:BGfinite}, and unlike in homogeneous spaces of finite or infinite volume, considered in Sec.~\ref{sec:BGhomogeneous}, we will not be able to write down a formula that can be evaluated just in terms of the spectrum of the Ricci quadratic form at each point. Instead, to evaluate our formula will require solving the geodesic equation. As such, it is not obviously easier to evaluate our formula than to calculate the volume of a geodesic ball directly. Nevertheless, for completeness we will proceed in this subsection by following the available mathematical techniques and seeing where they lead. \\

\noindent Let $M$ be a noncompact Riemannian manifold, now possibly of infinite volume, with the property that for all tangent vectors $\tau$, $|\mathrm{Ric}(\tau,\tau)|< c$, for a fixed constant $c$. At each point $q \in M$ define the family of probability measures $\Delta_q(t)$, $t \geq 0$, on $\R$. Points in $\R$ will be denoted by $x$.  $\Delta_q(0)$ is the probability (density) that a $\mu_{\text{sphere}}$-random unit vector $v$ at $q$ has $\mathrm{Ric}(v,v) = x$. We have written $\tau$ as $(q,v)$. For $t > 0$, $\Delta_q(t)$ is the probability density that $\mathrm{Ric}(\dot{\gamma}(t), \dot{\gamma}(t)) = x$ where $\gamma$ is the unit speed geodesic leaving $q$ in direction $v$. $\mu_{\text{sphere}}$ is the density on the unit tangent sphere to $q$ induced by the Riemannian metric. Let:
\begin{equation}
	\mathrm{cum}_{q,t}(x) = \int_{-\infty}^x \Delta_q(t)
\end{equation}
be the cumulative distribution. Define the measure, 
\begin{equation}
	\mathrm{cum}_q(x) = \sup_{t \in [0,\infty)}\mathrm{cum}_{q,t}(x) \ . 
\end{equation}

The supremums exist by compactness since $|\mathrm{Ric}(\tau,\tau)|$ is bounded over all of $M$. We note that $\mathrm{cum}(x)$ may be difficult to compute since the definition of $\mathrm{cum}_{q,t}(x)$ implicitly depends on the solution of the geodesic equations.

Define $\Delta_q(x) = \frac{d}{dx} \mathrm{cum}_q(x)$, the Radon-Nikodym derivative. Since $\Delta_q(x)$ integrates to one, it is a probability distribution. We may think of $\Delta_q(x)$ as solving the problem of constructing the density on $\R$ furthest to the right (i.e.~as positive as possible) so that for fixed $q \in M$ and $t \in [0,\infty)$ there is a measure-preserving map $h_{q,t}: \mathrm{support}(\Delta_{q,t}) \ra \mathrm{support}(\Delta_q)$ which is \emph{nonincreasing}, $h_{q,t}(x) \leq x$, and $h_{q,t^\ast}(\Delta_{q,t}) = \Delta_q$ i.e. the push forward of $\Delta_{q,t}$ is $\Delta_q$.

The point is that our chain of area estimates (line \ref{eq:thisequationhere123}) can be replaced with a pointwise estimate and augmented with an additional link to the right, valid for any radius $t_q  \leq t_{0,q}$, the focal distance:
\begin{equation}\label{chain}
\begin{split}
	A_q(t) & \coloneqq \int_{\mu_\text{sphere}} \det(J_\tau(t)) \leq \int_{\mu_\text{sphere}} \left(j_\tau(t)\right)^{d-1} \leq \int_{\mu_\text{sphere}} \left(j_\tau^{\text{Ric}}(t)\right)^{d-1} \leq \int_{\mu_\text{sphere}} \left(j_\tau^{h \text{Ric}}(t)\right)^{d-1} \ .
\end{split}
\end{equation}
The superscript $h\text{Ric}$ is short for solving the standard initial conditions Jacobi equations with coefficient
\begin{equation}
- 	\kappa(t) = h_{q,t}(\mathrm{Ric}(\dot{\gamma}(t),\dot{\gamma}(t)))\text{ where } q = \gamma(t), 
\end{equation}
where $\dot{\gamma}(t)$ is the tangent to the geodesic evolved from $\tau$ after time $t$. The nonincreasing property of $h_{q,t}$ guarantees the last inequality of line (\ref{chain}).

From here the argument, coefficient shuffling, is the same as in the compact case except having constructed the worst case distribution $\Delta_q$ allows us to do the average only over the compact unit sphere $S^{d-1}$ of the tangent space at any $q \in M$, instead of the noncompact unit tangent bundle. Because $M$ may have infinite volume no further integral over points $q$ is attempted. The result is recorded here as Theorem $2^\pr$; note that it contains some new information even for $M$ compact as it bounds the volume of all balls of radius $t$, not just the average.
\begin{customthm}{$\mathbf{2^\pr}$}
	Let $M^d$ be $d$ dimensional manifold, compact or noncompact, possibly with infinite volume, with $|\mathrm{Ric}(\tau,\tau) |\leq c$, with $c$ a constant. Then for any $q \in M$ the $\mathrm{Ball}_q(t)$ satisfies
	\[
	\textup{Area}(\de \mathrm{Ball}_q(t)) \leq \int_{\Delta_q} \textrm{sn}_k^{d-1}(t)
	\]
	where $k$ is drawn from $\Delta_q$ constructed above, for any  $t \leq t_0$, the focal distance. Similarly, $\mathrm{Vol}(B_q(t)) \leq \int_{\Delta_q} C_k^d(t)$, with notation as in Theorem 2. The integrals are over $\R $ with measure $\Delta_q$.
	\qed
\end{customthm}

By the same technique, Theorems \ref{thm:com1} and \ref{thm:com2} have corresponding generalizations to Theorems $\ref{thm:com1}^\pr$ and $\ref{thm:com2}^\pr$ in the infinite volume case. Let $M^d$ be any complete $d$-manifold.

\begin{customthm}{$\mathbf{3^\pr}$}
	For any $q \in M$
	\[
		1 \geq \frac{\int_{v \in T_q M} \ d\Omega \  {J}(r)}{\int_{k \in \R} \operatorname{sn}_k^{d-1}(r)} \geq \frac{\int_{v \in T_q M} \ d\Omega \  {J}(s)}{\int_{k \in \R } \operatorname{sn}_k^{d-1}(s)}
	\]
	for $0 < r < s < t_0$, the focal distance, the measure for the integrals in the denominators is $\Delta_q$. \qed
\end{customthm}

\begin{customthm}{$\mathbf{4^\pr}$}
	For any $q \in M$
	\[
		1 \geq \frac{\int_{S^{d-1}}  d \Omega \int_{t=0}^r  dt \, \det( {J}(t))}{\int_{k \in \R  } \Delta_q \int_{t=0}^r  dt \,  \operatorname{sn}_k^{d-1}(t)} \geq \frac{\int_{S^{d-1}}  d\Omega \int_{t=0}^s dt \, \det( {J}(t))}{\int_{k \in \R  } \Delta_q \int_{t=0}^s dt \, \operatorname{sn}_k^{d-1}(t)}
	\]
	for $0 < r < s < t_0$, the focal distance. \qed
\end{customthm}

\subsection{Enhanced-BG for homogeneous spaces}  \label{sec:BGhomogeneous} 

A homogeneous space is one in which all points are the same: given any two points $p$ and $q$ there is an isometry of the space taking $p$ to $q$. For homogeneous spaces, our result is 
\begin{thm}
The area of a geodesic sphere of radius $t$ in a homogeneous space is upperbounded by \emph{area(t)} $\leq$ \emph{{enhanced-BGarea}($t$)}. The volume of a geodesic ball of radius $t$ in a homogeneous space is upperbounded by \emph{volume(t)} $\leq$ \emph{{enhanced-BGvolume}($t$)}.  \label{theoremforhomogeneous}
\end{thm}
\noindent The upperbounds in these inequalities are defined by 
\begin{eqnarray}
\textrm{enhanced-BGarea}(t)
&\equiv &  \int d \Omega  \,  \textrm{sn}  \left(   \frac{ \mathcal{R}_{ \mu \nu } X_\Omega^\mu X_\Omega^\nu }{d-1} , t \right)^{d-1}   \\
\textrm{enhanced-BGvolume}(t) & \equiv & \int_0^t   d\tau \, \textrm{enhanced-BGarea}(\tau) . 
  \label{eq:explicitimproveBG}
\end{eqnarray}
\begin{proof}[Proof of Theorem 5]  Even without the homogeneity assumption on $\mathcal{M}$, geodesic flow has the following two properties: (i) it transports the Liouville measure without distortion, (ii) the push-forward of the Liouville measure from the cotangent bundle $\mathcal{T}^*\mathcal{M}$ of $\mathcal{M}$  to  $\mathcal{M}$ also transforms without distortion. (Note that the cotangent bundle and the tangent bundle are canonically isometric for Riemannian metrics: we can pass from a covariant tensor to a contravariant tensor using the metric.) Point (ii) can be seen by time-reversal symmetry: if there was a point $q$ in $\mathcal{M}$ at which the density was increased by geodesic flow in positive time, it would be decreased by time-reversed geodesic flow. But in the push-forward there is no difference between forward- and backwards-geodesic flow since all geodesics can be parameterized in both directions. Now exploit the fact that $\mathcal{M}$ is homogeneous to choose arbitrary isometry-induced identifications of the cotangent spaces at all points. Since the Liouville measure factors as the wedge product of the spatial measure and the tangential measure, constancy in the spatial factor now implies constancy of the tangential factor. This means that for any time $t$ the geodesic flow emanating from a single point $q$ has Ricci quadratic form values obeying the same law as the Ricci quadratic form values in any tangent space. This allows shuffling of the probability measure of these values on the sphere of radius $t$ about $q$, exactly as in the compact case. But now the shuffling is only between directions normal to the sphere. 
\end{proof}

\section{Exploring the homogeneous enhanced-BG theorem}
Let's examine the enhanced-BG bound for homogeneous spaces, Eq.~\ref{eq:newtheorem}. 
\subsection{Enhanced-BG theorem: example} \label{sec:improvedBGexamples}
In this subsection we look at the enhanced-BG theorem at work for an explicit example metric. We will pick a simple example for which we have an independent technique for calculating the volume exactly.  Since the example will be homogeneous, we can compare to the version of the enhanced-BG bound given in Eq.~\ref{eq:newtheorem}.

\subsubsection{$\mathbb{H}^2 \times \mathbb{R}^2$}
The metric for the product of a hyperbolic space and a plane is 
\begin{equation}
ds^2 = d\tau_1^2 +  \sinh^2 {\tau_1}{} \, d \phi^2  + d \tau_2^2 +  \tau_2^2\, d \psi^2.  \label{eq:H2R2metric}
\end{equation}
In this metric, geodesics do not `turn', meaning  geodesics that point partly in the $\mathbb{H}^2$ and partly in the $\mathbb{R}^2$ will continue to point in these directions  in unchanging proportion, so that $\mathcal{R}_{\mu \nu} X^{\mu} X^{\nu}$ is conserved along geodesics. \\

\noindent The Ricci curvature is diagonal in these coordinates, and given by 
\begin{equation}
\mathcal{R}^\mu_{\ \nu} = \textrm{diag}[ -1 , -1, \, 0, \, 0 \,] \ . 
\end{equation}
The Ricci curvature in a direction that points partly down one axis and partly down another is 
\begin{equation}
X^\mu = \{ \cos \theta , 0 , \sin \theta, 0 \} \ \ \rightarrow \ \ \ \mathcal{R}_{\mu \nu} X^\mu X^\nu =  - {\cos^2 \theta} \ . 
\end{equation}
The scalar curvature is 
\begin{equation}
\mathcal{R} = - 2 \  . \label{eq:ricciscalarnosurprise}
\end{equation}
\subsubsection{Exact volume for $\mathbb{H}^2 \times \mathbb{R}^2$}

 The volume contained within a geodesic ball of radius $t$ can be evaluated exactly:
\begin{equation}
\textrm{volume}(t) =   \int_0^{t} d \tau  \, 2 \pi   \sinh {\tau}{} \cdot \pi (t^2 - \tau^2) =   2 \pi^2  \left( 2  t \sinh t  - 2 \cosh t - t^2  + 2 \right) . \label{eq:thisformula234}
\end{equation}
(Since $\mathbb{H}^2 \times \mathbb{R}^2$ is a symmetric space, we could also have calculated the volume of geodesic balls using the general formula in \cite{symmetricspaces}.) 
\subsubsection{Bishop-Gromov bound for $\mathbb{H}^2 \times \mathbb{R}^2$} 
The BG bound can be evaluated by using  $\mathcal{R}_{\mu \nu} X_\textrm{min}^\mu  X_\textrm{min}^\nu =$ min$[\mathcal{R}^{\mu}_{\ \mu}] = -1$ in Eq.~\ref{eq:definitionofBGoft}, 
\begin{equation}
\textrm{BG}(t) = {24 \pi^2 }  \left(2 + \cosh \frac{t}{\sqrt{3} } \right) \sinh^4 \frac{t}{2\sqrt{3} }     . \label{eq:BGboundappliedtothisexample}
\end{equation}
\subsubsection{Enhanced-Bishop-Gromov bound for $\mathbb{H}^2 \times \mathbb{R}^2$} 
The enhanced-BG bound can also be evaluated. Let's look very close to the origin to  perform the angular integral. Writing $\tau_1 = t \cos \theta$ and $\tau_2 = t \sin \theta$ and expanding Eq.~\ref{eq:H2R2metric} for small $t$ gives 
 \begin{equation}
ds^2 = dt^2 + t^2 \left( d \theta^2 + \cos^2 \theta \, d \phi^2  + \sin^2 \theta \,  d \psi^2 \right) + \textrm{O}(t^4) + \ldots.  
\end{equation}
Since $\mathcal{R}_t^{\ t}$ depends only on $\theta$ and not on $\phi$ or $\psi$, the $\phi$ and $\psi$ integrals are trivial. This leaves  
\begin{equation}
\textrm{enhanced-BG}(t)  = (2 \pi)^2 \int_0^{\frac{\pi}{2}}  d \theta \cos \theta \sin \theta  \int_0^t d \tau \left(\frac{\sqrt{3} }{\cos \theta} \sinh \frac{\tau \cos \theta }{\sqrt{3} }\right)^3 . 
\end{equation}

\subsubsection{Comparing enhanced-BG  to exact volume}
The small-radius expansion is
\begin{eqnarray}
\textrm{volume}(t) & = & \frac{1}{2} \pi^2 t^4 \left( 1 + \frac{t^2}{18  } + \frac{t^4}{720  } + \ldots \right)  \label{eq:smalldistanceexact} \\
\textrm{enhanced-BG}(t) & = & \frac{1}{2} \pi^2 t^4 \left( 1 + \frac{t^2}{18  } + \frac{13 t^4}{6480  } + \ldots \right)  \label{eq:smalldistanceimprovedBG} \\
\textrm{BG}(t) & = &  \frac{1}{2} \pi^2 t^4 \left( 1 + \frac{t^2}{9  } + \frac{13 t^4}{2160  } + \ldots \right)  \ . \label{eq:smalldistanceBG}
\end{eqnarray}
The large-radius expansion is
\begin{eqnarray}
\textrm{volume}(t) & = & 2 \pi^2 \, t  \, \exp [ t ]  + \ldots \label{eq:exactlatetimegrowthH2R2} \\
\textrm{enhanced-BG}(t) & = & \frac{3}{2}  \pi^2 \,  t^{-1} \exp [ \sqrt{3} t ] + \ldots  \\
\textrm{BG}(t) & = & \frac{3}{4} \pi^2    \exp [\sqrt{3} t ]  + \ldots . 
\end{eqnarray}
As required by our theorems, at all times 
\begin{equation}
\mathbb{H}^2 \times \mathbb{R}^2: \ \ \textrm{volume}(t)  \leq \textrm{enhanced-BG}(t) \leq \textrm{BG}(t) . \label{eq:improvedboundhierarchy}
\end{equation}

     \begin{figure}[htbp] 
    \centering
    \includegraphics[width=\textwidth]{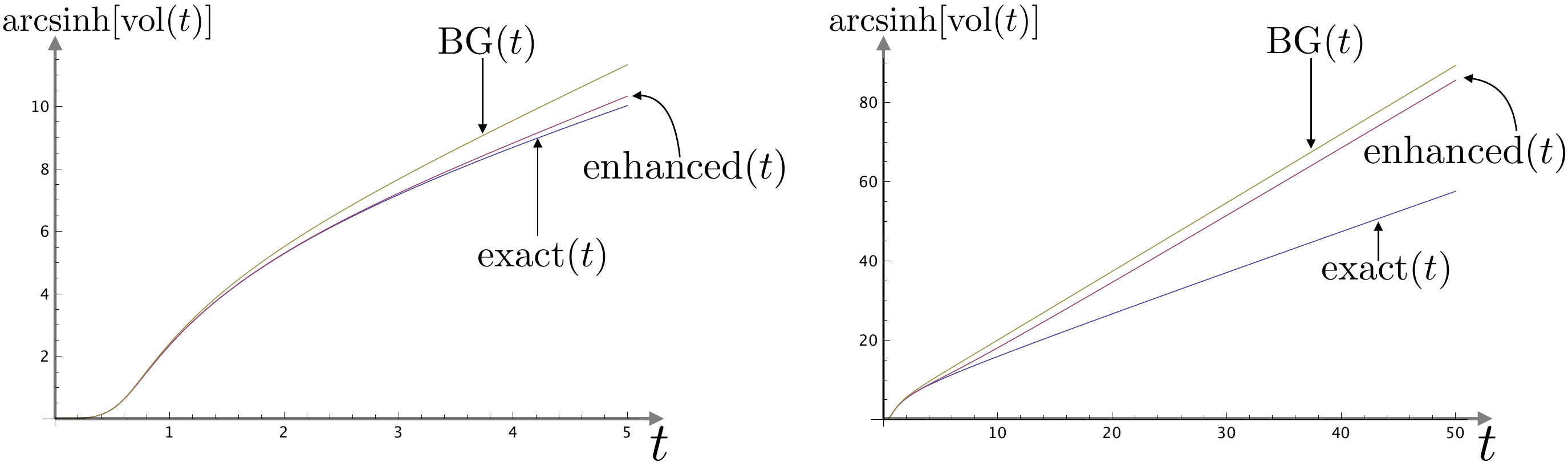} 
    \caption{The growth of volume (rescaled with an arcsinh for ease of comparison) for the example $\mathbb{H}^2 \times \mathbb{R}^2$. Left: at early times, the enhanced-Bishop-Gromov  bound is close to the exact answer. Right: at late times, the enhanced-Bishop-Gromov bound is only a little better than the unenhanced-BG bound. At all times volume$(t) \leq$ enhanced-BG$(t) \leq$ BG($t$).}
        \label{fig:improvedBGexample1}
 \end{figure}

\FloatBarrier

\noindent At short times, the enhanced-BG bound is \emph{much} enhanced over the BG bound---the O($t^2$) correction term is now exact. (In Sec.~\ref{eq:BBatshortdistances} we will show that the enhanced-BG bound always gets the O($t^2$) correction right, not just for this example.)  At late times, the enhanced-BG bound is \emph{barely} enhanced over the BG bound---the exponent isn't improved at all, and we just get an overall multiplicative improvement of $2/t$.

\subsection{Enhanced-BG at small radius} \label{eq:BBatshortdistances}
The enhanced-BG bound improves on the BG bound by being tighter (except for Einstein metrics, for which the two bounds are the same). Let's  look at  how much of an improvement we can expect for very small geodesic balls. 

Equations~\ref{eq:smalldistanceexact}-\ref{eq:smalldistanceBG} looked at the volume of a tiny geodesic ball for a particular example geometry. For that geometry, and Taylor-expanding in the radius, the enhanced-BG bound gave the {correct} O($t^2$) correction to the Euclidean expression (the enhanced-BG bound is tight at that order), whereas the BG bound did not (the BG bound is loose). In this subsection, we'll prove that the enhanced-BG bound is \emph{always} tight at O($t^2$) not just for that particular example but for all smooth metrics. We'll also explore the next-order correction. \\

\noindent Before we begin, let's do some preparatory Riemannology. Recall that $\mathcal{R}^2$, $\mathcal{R}_{\mu \nu} \mathcal{R}^{\mu \nu}$, \& $\mathcal{R}_{\mu \nu \rho \sigma} \mathcal{R}^{\mu \nu \rho \sigma}$ are all nonnegative quantities. However, those three do not exhaust the nonnegative quadratic quantities that we can construct out of the Riemann tensor. For example, we can break the Riemann tensor into irreducible representations of the orthogonal group by dividing it up into traceful $S_{\mu \nu \rho \sigma} $, semitraceless (Einstein) $E_{\mu \nu \rho \sigma}$, and completely traceless (Weyl) $C_{\mu \nu \rho \sigma}$ parts, as described in \cite{Decom} 
\begin{equation}
\mathcal{R}_{\mu \nu \rho \sigma} = S_{\mu \nu \rho \sigma} + E_{\mu \nu \rho \sigma} + C_{\mu \nu \rho \sigma} \ , 
\end{equation}
and then each of the individual pieces, when squared, give nonnegative scalars, 
\begin{eqnarray}
S_{\mu \nu \rho \sigma} S^{\mu \nu \rho \sigma} &=& \frac{2 \mathcal{R}^2 } {d(d-1)} \\
E_{\mu \nu \rho \sigma} E^{\mu \nu \rho \sigma} &=& \frac{4 \mathcal{R}_{\mu \nu} \mathcal{R}^{\mu \nu}} {d-2} - \frac{4 \mathcal{R}^2} {d(d-2)}   \label{eq:Esquaredinvaraint} \\ 
C_{\mu \nu \rho \sigma} C^{\mu \nu \rho \sigma} &=& \mathcal{R}_{\mu \nu \rho \sigma} \mathcal{R}^{\mu \nu \rho \sigma} - \frac{4 \mathcal{R}_{\mu \nu} \mathcal{R}^{\mu \nu}} {d-2} + \frac{2 \mathcal{R}^2} {(d-1)(d-2)}   \ . 
\end{eqnarray}

\subsubsection{Exact volume growth to O($t^4$)} 
The volume of a small geodesic ball of radius $t$ is \cite{Gray} 
\begin{equation}
\textrm{volume}(t) = \frac{\Omega_{d-1}}{d} t^d \left( 1 - \frac{ \mathcal{R}}{6(d+2)} t^2 + \frac{5 \mathcal{R}^2  +8\mathcal{R}_{\mu \nu}\mathcal{R}^{\mu \nu} - 3 \mathcal{R}_{\mu \nu \rho \sigma} \mathcal{R}^{\mu \nu \rho \sigma} - 18  \Box \mathcal{R}  }{360(d+2)(d+4)} t^4 + \ldots  \right)  \label{eq:exactvolumetofourthorderint}
\end{equation}
 For a homogeneous metric,  $\Box \mathcal{R} \equiv \nabla_\mu \nabla^\mu \mathcal{R} = 0$.

\subsubsection{Bishop-Gromov bound  to O($t^4$)} 
The Bishop-Gromov bound on the volume gives 
\begin{equation}
\textrm{BG}(t)  =  \frac{\Omega_{d-1}}{d} t^{d} \left( 1 - \frac{ d \, \textrm{min}[\mathcal{R^{\mu}_{\ \mu}} ] }{6(d+2)}   t^2 + \frac{ d (5d - 7) \, \textrm{min}[\mathcal{R^{\mu}_{\ \mu}} ]^2}{360(d-1)(d+4)}  t^4 + \ldots \right) \ .  \label{eq:BGgeneralouttofourthorder}
\end{equation}
The Bishop-Gromov bound gets the O($t^2$) term right if and only if 
\begin{equation}
\textrm{O}(t^2) \textrm{ BG is tight iff}: \ \ \ \mathcal{R} = d  \times  \textrm{min}[\mathcal{R^{\mu}_{\ \mu}}] \ , 
\end{equation}
which is to say iff every component of $\mathcal{R^{\mu}_{\ \mu}}$ is the same, so that the space is an `Einstein metric' with $\mathcal{R}_{\mu \nu} = \textrm{constant} \times g_{\mu \nu}$. This is equivalent to saying that the $E^2$ invariant from Eq.~\ref{eq:Esquaredinvaraint} is zero. 

For homogeneous spaces, subtracting Eq.~\ref{eq:exactvolumetofourthorderint} from Eq.~\ref{eq:BGgeneralouttofourthorder} tells us that the Bishop-Gromov bound gets both the O($t^2$) and O($t^4$) terms right if and only if 
\begin{equation}
E_{\mu \nu \rho \sigma} E^{\mu \nu \rho \sigma}  = C_{\mu \nu \rho \sigma} C^{\mu \nu \rho \sigma}  = 0 \ . \label{eq:EzeroWzero}
\end{equation}
These two conditions hold if and only if we have a maximally symmetric space.

 In summary, except for Einstein spaces,  the BG bound is loose at O($t^2$), and except for maximally symmetric spaces, the BG bound is loose at O($t^4$). 

\subsubsection{Enhanced-Bishop-Gromov bound  to O($t^4$)} 

In what follows we will need that\footnote{The indices follow by symmetry. The  prefactor  in Eq.~\ref{eq:Omegadeltaintegral1} follows from 
\begin{equation}
\Omega_{d-1} = \int d \Omega_{d-1} = \int d \Omega_{d-1}  \sum_{\mu=1}^d X_\Omega^{\mu} X_\Omega ^{\mu}  = d \times  \int d \Omega_{d-1} X^{\mu}_\Omega X_\Omega^{\mu} \Bigl|_\textrm{not summed} \  . 
\end{equation}
The  prefactor in Eq.~\ref{eq:Omegadeltaintegral2} follows from 
\begin{equation}
\Omega_{d-1} = \int d \Omega_{d-1} = \int d \Omega_{d-1} \left( \sum_{\mu=1}^d X_\Omega^{\mu} X_\Omega ^{\mu} \right)^2 =  \int d \Omega_{d-1} \left( d  \left( X^{\mu}_\Omega \right)^4 +  {d(d-1)}  \left(X_\Omega^{\mu}\right)^2 (X_\Omega^{\nu \neq \mu} )^2 \right) \Bigl|_\textrm{not summed} .
\end{equation}}, for unit vectors $X^{\mu}_{\Omega}$ (here `unit' means that $\sum_{\mu}^d  (X^{\mu}_{\Omega})^2 = 1$), 
\begin{eqnarray}
\int d \Omega_{d-1} \, X_\Omega^\mu X_\Omega^\nu& =&  \frac{1}{d} \, \Omega_{d-1} \,  \delta^{\mu \nu} \label{eq:Omegadeltaintegral1} \\
\int d \Omega_{d-1}  \, X_\Omega^\mu X_\Omega^\nu X_\Omega^\rho X_\Omega^\sigma & = &  \frac{1 }{d(d+2)} \Omega_{d-1}  \, \left( \delta^{\mu \nu} \delta^{\rho \sigma}  + \delta^{\mu \rho} \delta^{\nu \sigma}  + \delta^{\mu \sigma} \delta^{\nu \rho }  \right) \ . \label{eq:Omegadeltaintegral2}
\end{eqnarray}
Armed with these, let's expand out Eq.~\ref{eq:explicitimproveBG} at short distances,
\begin{eqnarray}
\textrm{enhanced-BG}(t)
 & = &  \int d \Omega_{d-1} \, \int_0^t d\tau  \  \textrm{sn} \left(   \frac{{ \mathcal{R}_{ \mu \nu } X_\Omega^\mu X_\Omega^\nu } }{{d-1}} , \tau  \right)^{d-1} \\
 & = &   \frac{t^{d}}{d} \int d \Omega_{d-1} \left( 1 - \frac{d \  \mathcal{R}_{ \mu \nu }  X_\Omega^\mu X_\Omega^\nu  }{6(d+2)}   t^2 + \frac{  d(5d - 7)(\mathcal{R}_{ \mu \nu } \mathcal{R}_{ \rho \sigma } X_\Omega^\mu X_\Omega^\nu X_\Omega^\rho X_\Omega^\sigma) }{360(d-1)(d+4)}  t^4 \ldots \right) \nonumber \\
 & = & \frac{\Omega_{d-1} t^d}{d} \left(1 - \frac{ \mathcal{R} }{6(d+2)}   t^2 +    \frac{  (5d - 7)(\mathcal{R}^2 + 2 \mathcal{R}_{ \mu \nu } \mathcal{R}^{ \mu \nu } ) }{360(d-1)(d+2)(d+4) }  t^4 \ldots \right)  \ \label{eq:improvedBGgeneralouttofourthorder} . 
\end{eqnarray}
Comparing to Eq.~\ref{eq:exactvolumetofourthorderint}, we see that the enhanced-Bishop-Gromov bound always gets the O($t^2$) correction right. \\

\noindent What about the O($t^4$) correction? For a homogeneous space, subtracting Eq.~\ref{eq:exactvolumetofourthorderint} from Eq.~\ref{eq:improvedBGgeneralouttofourthorder} gives 
\begin{equation}
\textrm{enhanced-BG}(t) - \textrm{vol}(t) = \frac{\Omega_{d-1} t^d}{d} \left(  \frac{d(1+d) E_{\mu \nu \rho \sigma} E^{\mu \nu \rho \sigma}  + 6(d-1) C_{\mu \nu \rho \sigma} C^{\mu \nu \rho \sigma} }{720(d-1)(2+d)(4+d)} t^4 + \textrm{O}(t^6) \right) \ . 
\end{equation}
This is nonnegative, as required by the enhanced-Bishop-Gromov theorem. The difference is zero if and only if 
\begin{equation}
E_{\mu \nu \rho \sigma} E^{\mu \nu \rho \sigma}  = C_{\mu \nu \rho \sigma} C^{\mu \nu \rho \sigma}  = 0 \ , \label{eq:nodifferenceonlyformaximallysymmetric}
\end{equation}
exactly as in Eq.~\ref{eq:EzeroWzero}. Thus the enhanced-Bishop-Gromov bound gets both the O($t^2$) and O($t^4$) terms right if and only if we have a maximally symmetric space. Otherwise the enhanced-BG bound is loose at O($t^4$). (That we got the same condition for the enhanced-BG bound as we did for the BG bound with an Einstein metric is no surprise---for an Einstein metric the enhanced-  and unenhanced-BG bounds are the same.)

\subsection{Enhanced-BG at large radius} \label{eq:BBatlongdistances}

 In this subsection, we'll see how much of an improvement the enhanced-BG bound can be over the BG bound for large geodesic balls.

When $\mathcal{R}^{\mu}_{\ \mu}$ is isotropic (i.e.~when we have an Einstein metric, so that all the eigenvalues of the Ricci curvature are degenerate), the enhanced-BG bound and the BG bound are the {same}.  If we want to look for a place where the two bounds are most \emph{different}, we should look for strongly anisotropic $\mathcal{R}^{\mu}_{\ \mu}$.  As an extreme example, let's consider 
\begin{equation}
\mathcal{R}^{\mu}_{\ \nu} = - (d-1) \delta^{\mu 1} \delta_{\nu 1} \ . \label{eq:onetrackRicci}
\end{equation}
All eigenvalues of the Ricci tensor are zero, except one eigenvalue is $\mathcal{R}^1_{\ 1} = -(d-1)$. (This is the value that $\mathcal{R}^1_{\ 1}$ would have in a unit hyperbolic space of dimension $d$.)  We're not going to worry about whether there actually is, as a matter of geometry, a metric that gives rise to this curvature tensor---our purpose is to investigate, as a matter of algebra, the maximum improvement that could conceivably be wrung from the enhanced-BG bound.

First let's evaluate the large-radius BG bound for this example. For large enough radius, we can replace \emph{sinh}'s with $\frac{1}{2}$\emph{exp}'s. Putting $\mathcal{R}_{\mu \nu}X^{\mu}_\textrm{min} X^{\nu}_\textrm{min} = -( d-1)$ into Eq.~\ref{eq:definitionofBGoft} gives
\begin{equation}
\textrm{BG}(t) =  \frac{\Omega_{d-1}}{d-1 }\left( \frac{1}{2} \exp [    t ] \right)^{d-1}  \left( 1   + \ldots \right) , 
\end{equation}
where the ``\ldots'' goes to zero at large $t$. 

Now let's evaluate the large-radius  enhanced-BG bound for this example. Define $\phi$ to be the angle between the bearing $\Omega$ and the negative-Ricci principal axis $\mu = 1$; the value of the Ricci curvature on this bearing is $\mathcal{R}_{\mu \nu} X^\mu_\phi X^\nu_\phi = -(d-1)\cos^2 \phi$.  At large radius Eq.~\ref{eq:explicitimproveBGversion2} gives  
\begin{eqnarray}
\textrm{enhanced-BG}(t) &=& \Omega_{d-2} \int_0^t d \tau \ 2 \int_0^{\frac{\pi}{2}} d\phi  \sin^{d-2} \phi  \left( \frac{\exp[ \cos \phi \ \tau] }{ 2 \cos \phi} \right)^{d-1} \left( 1 + \ldots \right) \\
& = &\textrm{BG}(t)  \frac{2 \, \Omega_{d-2}}{\Omega_{d-1}}  \int_0^{ \frac{\pi}{2} }  d\phi \,  \phi^{d-2}   \exp[  - \frac{1}{2} (d-1)  \phi^2  \, t] \left( 1 + \ldots \right)    , \label{eq:089obse}
\end{eqnarray} 
where in going from the first line to the second we have done an expansion in small $\phi$ (which is the dominant regime at large $t$). Again, the ``\ldots'' goes to zero at large $t$. 

We can straightforwardly evaluate this integral\footnote{Since we are prepared to tolerate multiplicative corrections ``\ldots'' that go to zero at large $t$, we can raise the upper limit of integration in Eq.~\ref{eq:089obse} from $\pi/2$ to $\infty$, and then perform that integration exactly
\begin{equation}
\frac{\textrm{enhanced-BG}(t)}{\textrm{BG}(t)} =  \frac{2}{\Omega_{d-1}} \left( \frac{2 \pi}{d-1} \frac{1}{t}  \right)^\frac{d-1}{2}  \left( 1 + \ldots \right) \ .  
\end{equation} }, but to qualitatively understand what is going on, consider which values of $\phi$ makes the largest contribution to the integrand. On the one hand, the geodesic with the greatest possible expansion is $\phi=0$, since that is the most negative direction for the Ricci curvature: the $\exp [\cos \phi \  t]$ term is largest for $\phi = 0$. On the other hand, only a tiny fraction of the directions leaving the origin have tiny $\phi$: the $\sin^{d-2} \phi$ term favors large values of $\phi$. The optimal tradeoff---and the largest contribution to the integrand---is at 
\begin{equation}
\partial_\phi \left( \phi^{d-2} \exp[  - \frac{1}{2} (d-1) \, \phi^2  t] \right)  = 0 \rightarrow \frac{d-2}{\phi_\textrm{max}} = (d-1) \phi_\textrm{max} t  \rightarrow \phi_\textrm{max}^2 = \frac{d-2}{d-1} \frac{1}{t} .
\end{equation}
As $t$ increases, the biggest contribution comes from a narrower and narrower beam closer and closer to the most-negative-Ricci axis.  For a bearing that lies within the beam, the contribution to the enhanced-BG bound is essentially undiminished from the contribution to the BG bound; on the other hand, bearings outside the beam barely contribute to the enhanced-BG bound at all. The enhanced-BG bound is thus smaller than the BG bound by the same ratio by which the beam is smaller than the full celestial sphere, namely roughly 
\begin{equation}
\frac{\textrm{enhanced-BG}(t)}{\textrm{BG}(t)} \sim \phi_\textrm{max}^{d-1} \sim \left( \frac{1}{t} \right) ^{\frac{d-1}{2}}.
\end{equation} 
At large $t$, the enhanced-BG bound is tighter than the BG bound by a large multiplicative factor. However, this multiplicative factor is only polynomially large in $t$, and so cannot compete with the overall exponential growth, 
\begin{eqnarray}
\textrm{BG}(t) &=& \exp[ (d-1) t + \textrm{O}(t^0) ]  \\
\textrm{enhanced-BG}(t) &=& \exp[ (d-1) t - \frac{1}{2}(d-1) \log t+  \textrm{O}(t^0) ]  \ . 
\end{eqnarray}
 As we saw in the example in Sec.~\ref{sec:improvedBGexamples}, the BG bound often fails to be tight by a multiplicative factor in the \emph{exponent}, and the enhanced-BG bound will fail to be tight by the same multiplicative factor. (Had we made the $\mu \neq 1$ Ricci curvatures positive, rather than zero as in Eq.~\ref{eq:onetrackRicci}, then the tightening from the enhanced-BG bound would be even stronger, but still not strong enough to correct the exponent.) Comparing to the results of Sec.~\ref{eq:BBatshortdistances}, we see the improvement rendered by moving from the BG to the enhanced-BG bound is more impressive for small balls than for large balls.

\subsection{Monotonicity results}
For homogeneous spaces, we have shown that 
\begin{equation}
\textrm{BG}(t) \geq \textrm{enhanced-BG}(t) \geq \textrm{volume}(t) \ . 
\end{equation}
Let's prove a slightly stronger result. We will show that for both of these inequalities, the difference between the two sides is monotonically non-decreasing. 

\subsubsection{Additive monotonicity for BG$(t)$ \emph{vs.}~eBG$(t)$} 
Let's prove that the additive factor by which the enhanced-BG bound (enhanced-BG$(t) \equiv $ eBG$(t)$) enhances the BG bound  (BG$(t)$)  is a monotonically non-decreasing function of $t$,
\begin{equation}
\textrm{additive monotonicity: } \ \ \ \ \frac{d}{dt} \left( \textrm{BG}(t) - \textrm{eBG}(t) \right)  \ \geq \ 0 \ . \label{eq:additivemontonicity}
\end{equation}
We can derive this by subtracting the enhanced-BG bound, Eq.~\ref{eq:explicitimproveBGversion2} from the BG bound, Eq.~\ref{eq:definitionofBGoft}, and differentiating with respect to time to give  
\begin{equation}
\frac{d}{dt} \left( \textrm{BG}(t) - \textrm{eBG}(t) \right) = \int d \Omega_{d-1} \left(  \operatorname{sn}\Bigl(  \frac{{ \mathcal{R}_{\mu \nu} X_\textrm{min}^{\mu} X_\textrm{min}^{\nu}}}{{d-1}}, \tau \Bigl)^{d-1} -  \operatorname{sn}\Bigl(  \frac{{ \mathcal{R}_{\mu \nu} X_{\Omega}^{\mu} X_{\Omega}^{\nu}}}{{d-1}}, \tau \Bigl)^{d-1}   \right) . \end{equation}
Since $\textrm{sn}(k,t)$ is a monotonically non-increasing function of $k$,  the integrand on the right-hand side is everywhere non-negative, establishing Eq.~\ref{eq:additivemontonicity}.
\subsubsection{Multiplicative monotonicity for BG$(t)$ \emph{vs.}~eBG$(t)$}  \label{subsec:multiplicativemonotonicityeBG}
Now let's prove that the multiplicative factor by which the enhanced-BG bound enhances the BG bound is also a monotonically non-decreasing function of $t$,
\begin{equation}
\textrm{multiplicative monotonicity: } \ \ \ \  \frac{d}{dt} \left(  \frac{\textrm{BG($t$)}}{\textrm{eBG($t$)}} \right) \ \geq \ 0 \ . \label{eq:multiplicativemonotonicity}
\end{equation}
If $\mathcal{R}_{\mu \nu}X^\mu_{\textrm{min}} X^\nu_{\textrm{min}} > 0$, then at some point  BG$(t)$ will stop growing, but Eq.~\ref{eq:additivemontonicity} guarantees that enhanced-BG($t$) will have stopped growing earlier, so in that era Eq.~\ref{eq:multiplicativemonotonicity} is trivially true (and saturated). Thus it only remains to prove Eq.~\ref{eq:multiplicativemonotonicity} in the era when BG$'(t) \equiv \textrm{BGarea}(t)>0$. It will be helpful to write the enhanced-BG bound, Eq.~\ref{eq:explicitimproveBGversion2},  in terms of the area element $\textrm{eBGarea}_\Omega(t)$  on the bearing $\Omega$, 
\begin{equation}
\textrm{eBG}(t) =  \int \hspace{-1mm} d t  \, \textrm{eBGarea}(t)  \equiv \int \hspace{-1mm} d t \hspace{-1mm} \int \hspace{-1mm} d \Omega \, \textrm{eBGarea}_\Omega(t)  \equiv \int \hspace{-1mm} dt \hspace{-1mm} \int \hspace{-1mm} d \Omega \, \textrm{sn}  \Bigl(  \frac{ {\mathcal{R}_{\mu \nu} X_\Omega^{\mu} X_\Omega^{\nu}}}{{d-1}}, t \Bigl)  ^{d-1} \ . 
\end{equation}
We can do the same rewriting for the BG bound, though in this case, and crucially for our proof, the area element $\textrm{BGarea}_\Omega(t)$ is the same for all bearings, 
\begin{equation}
\textrm{BG}(t) =  \int \hspace{-1mm} d t  \, \textrm{BGarea}(t)  \equiv \int \hspace{-1mm} d t \hspace{-1mm} \int \hspace{-1mm} d \Omega \, \textrm{BGarea}_\Omega(t)  \equiv \int \hspace{-1mm} d t \hspace{-1mm} \int \hspace{-1mm} d \Omega \, \textrm{sn}  \Bigl(  \frac{ {\mathcal{R}_{\mu \nu} X_\textrm{min}^{\mu} X_\textrm{min}^{\nu}}}{{d-1}}, t \Bigl)  ^{d-1}  \ . 
\end{equation}
To prove Eq.~\ref{eq:multiplicativemonotonicity}, first observe that $\textrm{BGarea}_\Omega(t)\geq \textrm{eBGarea}_\Omega(t) \geq 0$ and that the logarithmic derivative of the BG area element is larger than the logarithmic derivative of the eBG area element, 
\begin{equation} 
 \frac{\partial_t \textrm{BGarea}_\Omega(t)}{\textrm{BGarea}_\Omega(t)} \textrm{eBGarea}_\Omega(t) - \partial_t \textrm{eBGarea}_\Omega(t) \ \geq \ 0 \ .
\end{equation} 
This follows by direct calculation in terms of sn$(k,t)$, but we also recognize it as a special case (constant $\kappa$) of the monotonicity lemmas for Jacobi solutions derived in Sec.~\ref{sec:monotonicitylemmanegativekappa}. Next integrate with respect to $\Omega$ to pass from the area element to the full area 
\begin{equation} 
 \frac{\partial_t \textrm{BGarea}(t)}{\textrm{BGarea}(t)} \textrm{eBGarea}(t) - \partial_t \textrm{eBGarea}(t) \ \geq \ 0 \ . \label{eq:montoneareaenhancedment}
\end{equation} 
(Note that this only works because BGarea$_\Omega(t)$ is independent of $\Omega$.) This is a monotonicity result on the logarithmic derivative of the area. To turn this into a monotonicity result on the logarithmic derivative of the volume, consider the quantity 
\begin{equation}
\Lambda =  \frac{\partial_t \textrm{BG}(t)}{\textrm{BG}(t)} \textrm{eBG}(t) - \partial_t \textrm{eBG}(t) \ . 
\end{equation}
Let's show that $\Lambda$ never goes negative. On the one hand, $\Lambda$ starts off non-negative. And on the other hand as a matter of algebra 
\begin{equation}
\frac{d \Lambda}{dt} =   \Lambda \left( \frac{\partial_t  \textrm{BGarea}(t)}{   \textrm{BGarea}(t)}   - \frac{ \partial_t  \textrm{BG}(t)}{\textrm{BG}(t)} \right)  + \left(  \frac{\partial_t  \textrm{BGarea}(t)}{   \textrm{BGarea}(t)}  \textrm{eBGarea}(t)   - {\partial_t  \textrm{eBGarea}(t)} \right)  . 
\end{equation}
Using Eq.~\ref{eq:montoneareaenhancedment}, this tells us that the derivative of $\Lambda$ when $\Lambda = 0$ is always non-negative, so $\Lambda$ cannot  cross zero and go negative. This establishes Eq.~\ref{eq:multiplicativemonotonicity}. 

\subsubsection{Additive monotonicity for eBG$(t)$ \emph{vs.}~volume$(t)$} 
It follows from the enhanced-BG bound that enhanced-BG$(t) \geq \textrm{volume}(t)$. Let's show that, further, the additive gap never decreases, 
\begin{equation}
\textrm{additive monotonicity: } \ \ \ \ \frac{d}{dt} \left( \textrm{eBG}(t) - \textrm{volume}(t) \right)  \ \geq \ 0 \ . \label{eq:additivemontonicityvolume} \ 
\end{equation}
This follows directly from Theorem~\ref{theoremforhomogeneous}, which established that 
\begin{equation}
\textrm{eBGarea}(t) \geq \textrm{area}(t)  \ . 
\end{equation}
\subsubsection{Multiplicative monotonicity for eBG$(t)$ \emph{vs.}~volume$(t)$?} 
It was established by Gromov that 
\begin{equation}
\frac{d}{dt} \left( \frac{\textrm{BG}(t)}{\textrm{volume}(t)} \right) \ \geq  \ 0 \ . 
\end{equation}
Let us note that we have \emph{not} established an analogous result for the enhanced-BG bound, 
\begin{equation}
\textrm{not proved:} \ \ \frac{d}{dt} \left( \frac{\textrm{eBG}(t)}{\textrm{volume}(t)} \right) \ \geq \ 0 \ . 
\end{equation}
We cannot prove this using the technique of Sec.~\ref{subsec:multiplicativemonotonicityeBG} because {both} eBGarea$_\Omega(t)$ {and} area$_\Omega(t)$ are functions of $\Omega$, unlike BGarea$_\Omega(t)$.

\section{Conclusion} 
The Bishop-Gromov bound for homogeneous spaces has to make a worst-case assumption: it treats every geodesic  as though it points in the most expansive direction $X^\mu_\textrm{min}$. The enhanced-Bishop-Gromov bound for homogeneous spaces replaces this worst-case assumption with an \emph{average} over directions, but you have to be careful how you take the average. If you do an unweighted average of the Ricci curvature over all directions leaving the starting point, this just gives the Ricci scalar 
\begin{equation}
 \int d \Omega \, \mathcal{R}_{\mu \nu} X^\mu_\Omega X^\nu_\Omega = \Omega (d-1) \mathcal{R},
\end{equation}
and as we discussed (and will see again in more detail in the appendix), no interesting upperbound can be placed on the volume using just the Ricci scalar, even if we restrict to spaces of negative curvature. Since volume grows exponentially fast in negatively curved spaces, you might summarize this with the slogan that you can't `average  \emph{then} exponentiate'. Instead, the enhanced Bishop-Gromov bound for homogeneous spaces, Eq.~\ref{eq:explicitimproveBGversion2}, takes a time-dependent \emph{weighted} average of the Ricci curvature: the enhanced-BG bound exponentiates  \emph{then} averages. 

For inhomogeneous spaces, we were unable to place a tighter upperbound on the quantity bounded by the original BG-bound: we were unable to improve the upperbound on the rate of growth of geodesic balls around the \emph{worst-case} starting point $q$. However, for finite-volume inhomogeneous spaces we were able to place a novel upperbound on a different quantity: the \emph{average} rate of growth of geodesic balls, averaged over all starting points. Once again, to derive the bound we needed to exponentiate {then} average. \\

\noindent Though the enhanced-BG bound is tighter than the BG bound, it is still typically not tight. Let's discuss the three  effects that may prevent the enhanced-BG bound being tight for homogeneous spaces. 
\begin{enumerate}
\item Cut loci. 

The volume of a geodesic ball counts each point at most once: the point counts if it has been visited by a geodesic, and doesn't count if it hasn't been visited. No point counts more than once, even if it has been visited by more than one geodesic. Thus, once you hit a cut locus---as you might for example in a compact hyperbolic space---the BG and enhanced-BG bounds may start double-counting the volume, and therefore no longer be tight.

\item Turning. 

`Turning' is what we call it when $\mathcal{R}_{\mu \nu}X^\mu X^\nu$ is not conserved along a geodesic: the geodesic `turns' towards a direction of different Ricci curvature. This leads to the direction of fastest acceleration being shared amongst multiple geodesics, which as we saw in Sec.~\ref{sec:correlatedJacobis} is inefficient: the net effect of turning is to \emph{reduce}  the total rate of growth. Non-conservation of $\kappa(t) \sim - \mathcal{R}_{\mu \nu}X^\mu X^\nu$ leads to the enhanced-BG bound being loose.

\item Shear. 

In the Raychaudhuri equation, Eq.~\ref{eq:Raychaudhuri}, we saw that shear $\sigma^2$ will slow the growth of volume, making the enhanced-BG bound loose.

Shear is when the directions orthogonal to the geodesic are expanding at different rates, and so is sourced by having unequal sectional curvatures. An example of a space that gives rise to shear is $\mathbb{H}^2 \times \mathbb{R}^2$,  considered in Sec.~\ref{sec:improvedBGexamples}.  The space $\mathbb{H}^2 \times \mathbb{R}^2$ has unequal sectional curvatures, since the section that spans the $\mathbb{H}^2$ has negative sectional curvature, whereas the sections that have at least one leg down the $\mathbb{R}^2$ have zero sectional curvature. Since this example has neither cut loci nor turning, for this example the entire reason  the enhanced-BG bound fails to be tight is shear. 

We can intuitively understand why shear slows the rate of growth. 
 In the shuffling lemma, Eq.~\ref{eq:shufflinglemma}, we saw that the \emph{sum} of a collection of Jacobi solutions was maximized by shuffling their coefficients so as to \emph{increase} the inequality between them. However, looking ahead \cite{toappear}, it is easy to prove a complementary theorem.  In order to maximize the \emph{product} of the solutions, we should shuffle the coefficients so as a \emph{minimize} the inequality between them; if we shuffle fast enough then every trajectory effectively follows the average schedule, and the relevant theorem is 
\begin{equation}
\forall_i, j_i \geq 0 \ \rightarrow \ \prod_i j_\textrm{av.}(t) \ \geq \  \prod_i j_i(t) \ \ \ \textrm{where } \ j''_\textrm{av.}(t) \equiv  \kappa_{\textrm{av.}}(t) j_\textrm{av.}(t) \ , \label{eq:productheorem}
\end{equation}
where $\kappa_\textrm{av.}$ is the average of the schedules $\sum_i \kappa_\textrm{av.}(t) \equiv \sum_i \kappa_i(t)$.   The area element in the direction of a geodesic is determined by the determinant of the Jacobian, which is simply the product of the eigenvalues. Eq.~\ref{eq:productheorem} then implies that, at fixed $\mathcal{R}_{\mu \nu}X^\mu X^\nu$, the growth rate is  maximized when all the sections in which that geodesic participates are expanding at the same rate. 
This is equivalent to putting the shear to zero.  

We will have much more to say about this in a forthcoming paper \cite{toappear}, where we will show how to use higher curvature invariants to lowerbound the shear and further enhance the BG bound. 

\end{enumerate}

\section*{Acknowledgements}
MHF thanks the Aspen Center for Physics for hospitality.

\appendix

\section{Counterexample to   $\textrm{vol}(t) \leq \textrm{vol}_{\mathbb{H} \mathcal{[R]}}(t)$} \label{sec:H3R2counterexample}
One of the attractive properties of the Bishop-Gromov bound, and of the enhanced-Bishop-Gromov bound, is that we are not required to know the full four-index Riemann tensor $\mathcal{R}_{\mu \nu \rho \sigma}$, but only the contracted two-index Ricci tensor $\mathcal{R}_{\mu \nu} = \mathcal{R}_{\mu \nu \rho}^{\ \ \ \ \rho}$. It it non-trivial that such a bound should be possible, since the divergence of geodesics is governed by the sectional curvatures,  and the sectional curvatures are determined by $\mathcal{R}_{\mu \nu \rho \sigma}$ but not  $\mathcal{R}_{\mu \nu}$. But it turns out that just knowing the Ricci curvature---i.e. just knowing the \emph{average} sectional curvature---is enough to get a bound. (The heuristic explanation for why this works is that variance in the sectional curvatures generates shear, and shear in the Raychaudhuri equation Eq.~\ref{eq:Raychaudhuri} can only decrease the rate of expansion, so pretending that all sectional curvatures are given by the average sectional curvature will only lead one to \emph{over}estimate the volume, not underestimate it; for more see \cite{toappear}.) 

One might wonder whether we can take this one step further, and develop a useful bound that depends only on the Ricci scalar,  $\mathcal{R} = \mathcal{R}_{\mu}^{\ \mu}$. In particular, for a homogeneous space one might wonder whether there is a bound of the form 
\begin{equation}
?? \ \ \textrm{volume}(t) \ \leq \ \textrm{volume}_{\mathbb{H}[\mathcal{R}]}(t) \ \ ?? \ , \label{eq:wrongconjecture}
\end{equation}
where $\textrm{volume}_{\mathbb{H}[\mathcal{R}]}(t) $ is defined as the volume of a geodesic ball in the maximally symmetric space with the same (dimension and) Ricci scalar. Let's show that this doesn't work. 

If we allow any sign of the curvature, it is obvious this is not going to work. Consider the space $\mathbb{H}^2 \times \mathbb{S}^2$. If we make the radius of curvature of the $\mathbb{S}^2$ smaller than the radius of curvature of the  $\mathbb{H}^2$, this space will have positive Ricci curvature. The maximally symmetric space of the same Ricci curvature will be an $\mathbb{S}^4$, which has finite volume. On the other hand, because of the $\mathbb{H}^2$, the volume of a geodesic ball in $\mathbb{H}^2 \times \mathbb{S}^2$ without bound.  

\subsection{Explicit example: $\mathbb{H}^3 \times \mathbb{R}^2$} 

It is less trivial that this bound isn't going to work if we restrict the curvature to be non-positive, but let's show that now. Consider the metric\footnote{We have to go to $\mathbb{H}^3 \times \mathbb{R}^2$, because for the special case of $\mathbb{H}^2 \times \mathbb{R}^2$ considered in Sec.~\ref{sec:improvedBGexamples}, the bound Eq.~\ref{eq:wrongconjecture} actually \emph{does} apply, since by direct computation $\textrm{volume}(t) \leq \textrm{vol}_{\mathbb{H}[\mathcal{R}]}(t) = {96 \pi^2 } (2 + \cosh [ \sqrt{ \frac{ 1 }{6 } }  {t}] ) 
\sinh^4 [ \frac{1}{2} \sqrt{ \frac{ 1 }{6 } }  {t}] \leq \textrm{enhanced-BG}(t)$.} 
$\mathbb{H}^3 \times \mathbb{R}^2$: 
\begin{equation}
ds^2 = d\tau_1^2 + \tau_1^2 \, d \phi^2  + d \tau_2^2 +   \sinh^2 {\tau_2} \left( d \psi^2 + \sin^2 \psi d \chi^2\right).  \label{eq:R2H3metric}
\end{equation}
This space has scalar curvature
\begin{equation}
\mathcal{R} = - 6 \ . 
\end{equation}
The exact volume of a geodesic ball of radius $t$ is 
\begin{eqnarray}
\textrm{volume}(t) &=&  \int_0^{t} d \tau  \, 4 \pi  \sinh^2 {\tau} \cdot \pi (t^2 - \tau^2) \\
& = & \frac{1}{6} \pi ^2 \left(-8 t^3-3 \sinh [2 t]+6 t \cosh [2 t]\right) \ . \label{eq:thisformula092834}
\end{eqnarray}
The volume of a geodesic ball in the five-dimensional hyperbolic space with the same value of $\mathcal{R}$ is given by
\begin{eqnarray}
\textrm{volume}_{\mathbb{H[\mathcal{R}]}}(t)   & = &  \Omega_{d-1} \int_0^t d\tau \left( \sqrt{ \frac{d(d-1)}{- \mathcal{R} }} \sinh \left[   \frac{\sqrt{- \mathcal{R}} \tau}{\sqrt{d(d-1)}}  \right] \right)^{d-1} \\
& = & \frac{8 \pi^2}{3} \int_0^t d\tau \left( \sqrt{ \frac{10}{3 } } \sinh \left[ \sqrt{ \frac{3 }{10}} \tau  \right] \right)^4 \\
& = & \frac{25}{81} \pi ^2 \left(36 t+\sqrt{30} \left(\sinh \Bigl[2 \sqrt{\frac{6}{5}}
   t\Bigl]-8 \sinh \Bigl[\sqrt{\frac{6}{5}} t\Bigl] \right)\right) \ .  \label{eq:thisotherformula} 
  \end{eqnarray}
Comparing Eqs.~\ref{eq:thisformula092834} and \ref{eq:thisotherformula}, we find 
\begin{equation}
0 < t < 7.3216 \ \ \ \rightarrow \ \ \ \textrm{volume}(t) > \textrm{volume}_{\mathbb{H[R]}}(t)  \ . 
\end{equation} 
This is a counterexample to the conjecture Eq.~\ref{eq:wrongconjecture}. Deforming the $\mathbb{R}^2$ to be an $\mathbb{H}^2$ of very large curvature length, this also serves as a counterexample to the conjecture that Eq.~\ref{eq:wrongconjecture} might hold for spaces of strictly negative Ricci curvature. 

\subsection{Taylor expanding general metric} 

The volume of a small geodesic ball of radius $t$ in a maximally symmetric space of Ricci scalar $\mathcal{R}$ is 
\begin{equation}
\textrm{volume}_{\mathbb{H}\mathcal{[R]}}(t) =
  \frac{\Omega_{d-1}}{d} t^{d} \left( 1 - \frac{ \mathcal{R} }{6(d+2)}   t^2 + \frac{  (5d - 7)\mathcal{R}^2}{360(d-1)d(d+4)}  t^4 \ldots \right) \ . 
  \end{equation}
Comparing this to the volume of a small ball in the original space, Eq.~\ref{eq:exactvolumetofourthorderint}, we see that it gets the O($t^2$) correction term exactly right, but deviates at O($t^4$), 
\begin{equation}
\textrm{vol}(t) - \textrm{vol}_{\mathbb{H}\mathcal{[R]}}(t) =   \frac{\Omega_{d-1}}{d} t^{d} \left( \frac{(7 - 2d) E_{\mu \nu \rho \sigma} E^{\mu \nu \rho \sigma} + 3 C_{\mu \nu \rho \sigma} C^{\mu \nu \rho \sigma}}{360(2+d)(4+d)}t^4 + O(t^6) \right) \ . 
\end{equation}
For $d \geq 4$ this will go negative for large $E^2$  and small $C^2$, as we saw with $\mathbb{H}^3 \times \mathbb{R}^2$. 

\subsection{Unfalsified possibilities}
Our examples and Taylor expansion do not eliminate the possibility that Eq.~\ref{eq:wrongconjecture} might hold  for all negatively curved homogeneous spaces \emph{at late times} (indeed, it is easy to check that no product of negatively curved maximally symmetric spaces of any dimensions or curvature lengths would violate Eq.~\ref{eq:wrongconjecture} at late times); nor that Eq.~\ref{eq:wrongconjecture} might hold for all homogeneous $d = 3$ spaces at all times.

\end{document}